\documentclass[10pt]{article}
\usepackage{amsmath}
\usepackage{amssymb}
\usepackage{amsthm}
\usepackage{mathrsfs}
\usepackage[compress,noadjust]{cite}
\numberwithin{equation}{section}

\begin{document}
\title{Weighted inequalities for fractional integral operators and linear commutators in the Morrey type spaces}
\author{Hua Wang \footnote{E-mail address: wanghua@pku.edu.cn.}\\
\footnotesize{College of Mathematics and Econometrics, Hunan University, Changsha 410082, P. R. China}}
\date{}
\maketitle

\begin{abstract}
In this paper, we first introduce some new Morrey type spaces containing generalized Morrey space and weighted Morrey space with two weights as special cases. Then we give the weighted strong type and weak type estimates for fractional integral operators $I_\alpha$ in these new Morrey type spaces. Furthermore, the weighted strong type estimate and endpoint estimate of linear commutators $[b,I_{\alpha}]$ formed by $b$ and $I_{\alpha}$ are established. Also we study related problems about two-weight, weak type inequalities for $I_{\alpha}$ and $[b,I_{\alpha}]$ in the Morrey type spaces and give partial results.
\end{abstract}

\section{Introduction}

For given $\alpha$, $0<\alpha<n$, the fractional integral operator (or the Riesz potential) $I_{\alpha}$ of order $\alpha$ is defined by
\begin{equation*}
I_{\alpha}f(x):=\frac{1}{\gamma(\alpha)}\int_{\mathbb R^n}\frac{f(y)}{|x-y|^{n-\alpha}}\,dy,
\quad\mbox{and}\quad \gamma(\alpha)=\frac{\pi^{\frac{n}{\,2\,}}2^\alpha\Gamma(\frac{\alpha}{\,2\,})}{\Gamma(\frac{n-\alpha}{2})}.
\end{equation*}
It is well-known that the Hardy--Littlewood--Sobolev theorem states that the fractional integral operator $I_{\alpha}$ is bounded from $L^p(\mathbb R^n)$ to $L^q(\mathbb R^n)$ for $0<\alpha<n$, $1<p<n/{\alpha}$ and $1/q=1/p-{\alpha}/n$. Also we know that $I_{\alpha}$ is bounded from $L^1(\mathbb R^n)$ to $WL^q(\mathbb R^n)$ for $0<\alpha<n$ and $q=n/{(n-\alpha)}$ (see \cite{stein}). In 1974, Muckenhoupt and Wheeden \cite{muckenhoupt} studied the weighted boundedness of $I_{\alpha}$ and obtained the following results.

\newtheorem{theorem}{Theorem}[section]
\newtheorem{defn}{Definition}[section]
\newtheorem{corollary}{Corollary}[section]
\newtheorem{lemma}{Lemma}[section]

\begin{theorem}[\cite{muckenhoupt}]\label{strong}
Let $0<\alpha<n$, $1<p<n/{\alpha}$, $1/q=1/p-{\alpha}/n$ and $w\in A_{p,q}$. Then the fractional integral operator $I_{\alpha}$ is bounded from $L^p(w^p)$ to $L^q(w^q)$.
\end{theorem}

\begin{theorem}[\cite{muckenhoupt}]\label{weak}
Let $0<\alpha<n$, $p=1$, $q=n/{(n-\alpha)}$ and $w\in A_{1,q}$. Then the fractional integral operator $I_{\alpha}$ is bounded from $L^1(w)$ to $WL^q(w^q)$.
\end{theorem}

For $0<\alpha<n$, the linear commutator $[b,I_{\alpha}]$ generated by a suitable function $b$ and $I_{\alpha}$ is defined by
\begin{equation*}
\begin{split}
[b,I_\alpha]f(x)&:=b(x)\cdot I_\alpha f(x)-I_\alpha(bf)(x)\\
&=\frac{1}{\gamma(\alpha)}\int_{\mathbb R^n}\frac{[b(x)-b(y)]\cdot f(y)}{|x-y|^{n-\alpha}}\,dy.
\end{split}
\end{equation*}

In 1991, Segovia and Torrea \cite{segovia} proved that $[b,I_{\alpha}]$ is also bounded from $L^p(w^p)$ ($1<p<n/{\alpha}$) to $L^q(w^q)$ whenever $b\in BMO(\mathbb R^n)$ (see also \cite{chanillo} for the unweighted case).

\begin{theorem}[\cite{segovia}]\label{cstrong}
Let $0<\alpha<n$, $1<p<n/{\alpha}$, $1/q=1/p-{\alpha}/n$ and $w\in A_{p,q}$. Suppose that $b\in BMO(\mathbb R^n)$, then the linear commutator $[b,I_{\alpha}]$ is bounded from $L^p(w^p)$ to $L^q(w^q)$.
\end{theorem}

In 2007, Cruz-Uribe and Fiorenza \cite{cruz5} discussed the weighted endpoint inequalities for commutator of fractional integral operator and proved the following result (see also \cite{cruz4} for the unweighted case).

\begin{theorem}[\cite{cruz5}]\label{cweak}
Let $0<\alpha<n$, $p=1$, $q=n/{(n-\alpha)}$ and $w^q\in A_1$. Suppose that $b\in BMO(\mathbb R^n)$, then for any given $\sigma>0$ and any bounded domain $\Omega\subset\mathbb R^n$, there is a constant $C>0$, which does not depend on $f$, $\Omega$ and $\sigma>0$, such that
\begin{equation*}
\Big[w^q\big(\big\{x\in\Omega:\big|[b,I_\alpha](f)(x)\big|>\sigma\big\}\big)\Big]^{1/q}
\leq C\int_{\Omega}\Phi\left(\frac{|f(x)|}{\sigma}\right)\cdot w(x)\,dx,
\end{equation*}
where $\Phi(t)=t\cdot(1+\log^+t)$ and $\log^+t=\max\{\log t,0\}$.
\end{theorem}

On the other hand, the classical Morrey space was originally introduced by Morrey in \cite{morrey} to study the local behavior of solutions to second order elliptic partial differential equations. This classical space and various generalizations on the Euclidean space $\mathbb R^n$ have been extensively studied by many authors. In \cite{mizuhara}, Mizuhara introduced the generalized Morrey space $\mathcal L^{p,\Theta}(\mathbb R^n)$ which was later extended and studied in \cite{nakai}. In \cite{komori}, Komori and Shirai defined a version of the weighted Morrey space $\mathcal L^{p,\kappa}(v,u)$ which is a natural generalization of the weighted Lebesgue space.

Let $I_\alpha$ be the fractional integral operator, and let $[b,I_{\alpha}]$ be its linear commutator. The main purpose of this paper is twofold. We first define a new kind of Morrey type spaces $\mathcal M^{p,\theta}(v,u)$ containing generalized Morrey space $\mathcal L^{p,\Theta}(\mathbb R^n)$ and weighted Morrey space $\mathcal L^{p,\kappa}(v,u)$ as special cases. As the Morrey type spaces may be considered as an extension of the weighted Lebesgue space, it is natural and important to study the weighted boundedness of $I_{\alpha}$ and $[b,I_{\alpha}]$ in these new spaces. And then we will establish the weighted strong type and endpoint estimates for $I_{\alpha}$ and $[b,I_{\alpha}]$ in these Morrey type spaces $\mathcal M^{p,\theta}(v,u)$ for all $1\leq p<\infty$. In addition, we will discuss two-weight, weak type norm inequalities for $I_\alpha$ and $[b,I_{\alpha}]$ in $\mathcal M^{p,\theta}(v,u)$ and give some partial results.

\section{Statements of the main results}

\subsection{Notations and preliminaries}
Let $\mathbb R^n$ be the $n$-dimensional Euclidean space of points $x=(x_1,x_2,\dots,x_n)$ with norm $|x|=(\sum_{i=1}^n x_i^2)^{1/2}$. For $x_0\in\mathbb R^n$ and $r>0$, let $B(x_0,r)=\{x\in\mathbb R^n:|x-x_0|<r\}$ denote the open ball centered at $x_0$ of radius $r$, $B(x_0,r)^c$ denote its complement and $|B(x_0,r)|$ be the Lebesgue measure of the ball $B(x_0,r)$. A non-negative function $w$ defined on $\mathbb R^n$ is called a weight if it is locally integrable. We first recall the definitions of two weight classes; $A_p$ and $A_{p,q}$.
\begin{defn}[$A_p$ weights \cite{muckenhoupt1}]
A weight $w$ is said to belong to the class $A_p$ for $1<p<\infty$, if there exists a positive constant $C$ such that for any ball $B$ in $\mathbb R^n$,
\begin{equation*}
\left(\frac1{|B|}\int_B w(x)\,dx\right)^{1/p}\left(\frac1{|B|}\int_B w(x)^{-p'/p}\,dx\right)^{1/{p'}}\leq C<\infty,
\end{equation*}
where $p'$ is the dual of $p$ such that $1/p+1/{p'}=1$. The class $A_1$ is defined replacing the above inequality by
\begin{equation*}
\frac1{|B|}\int_B w(x)\,dx\leq C\cdot\underset{x\in B}{\mbox{ess\,inf}}\;w(x),
\end{equation*}
for any ball $B$ in $\mathbb R^n$. We also define $A_\infty=\bigcup_{1\leq p<\infty}A_p$.
\end{defn}

\begin{defn}[$A_{p,q}$ weights \cite{muckenhoupt}]
A weight $w$ is said to belong to the class $A_{p,q}$ $(1<p,q<\infty)$, if there exists a positive constant $C$ such that for any ball $B$ in $\mathbb R^n$,
\begin{equation*}
\left(\frac1{|B|}\int_B w(x)^q\,dx\right)^{1/q}\left(\frac1{|B|}\int_B w(x)^{-p'}\,dx\right)^{1/{p'}}\leq C<\infty.
\end{equation*}
The class $A_{1,q}$ $(1<q<\infty)$ is defined replacing the above inequality by
\begin{equation*}
\left(\frac1{|B|}\int_B w(x)^q\,dx\right)^{1/q}\left(\underset{x\in B}{\mbox{ess\,sup}}\;\frac{1}{w(x)}\right)\leq C<\infty.
\end{equation*}
\end{defn}

\begin{lemma}\label{relation}
Suppose that $0<\alpha<n$, $1\leq p<n/{\alpha}$ and $1/q=1/p-{\alpha}/n$. The following statements are true (see \cite{lu}):

(i) If $p>1$, then $w\in A_{p,q}$ implies $w^q\in A_q$ and $w^{-p'}\in A_{p'};$

(ii) If $p=1$, then $w\in A_{1,q}$ if and only if $w^q\in A_1$.
\end{lemma}
Given a ball $B$ and $\lambda>0$, $\lambda B$ denotes the ball with the same center as $B$ whose radius is $\lambda$ times that of $B$. For a given weight function $w$ and a Lebesgue measurable set $E$, we denote the characteristic function of $E$ by $\chi_E$, the Lebesgue measure of $E$ by $|E|$ and the weighted measure of $E$ by $w(E)$, where $w(E):=\int_E w(x)\,dx$. Given a weight $w$, we say that $w$ satisfies the doubling condition if there exists a universal constant $C>0$ such that for any ball $B$ in $\mathbb R^n$, we have
\begin{equation}\label{weights}
w(2B)\leq C\cdot w(B).
\end{equation}
When $w$ satisfies this doubling condition \eqref{weights}, we denote $w\in\Delta_2$ for brevity. We know that if $w$ is in $A_{\infty}$, then $w\in\Delta_2$ (see \cite{garcia}). Moreover, if $w\in A_\infty$, then for any ball $B$ and any measurable subset $E$ of $B$, there exists a number $\delta>0$ independent of $E$ and $B$ such that (see \cite{garcia})
\begin{equation}\label{compare}
\frac{w(E)}{w(B)}\le C\left(\frac{|E|}{|B|}\right)^\delta.
\end{equation}

Given a weight function $w$ on $\mathbb R^n$, for $1\leq p<\infty$, the weighted Lebesgue space $L^p(w)$ is defined as the set of all functions $f$ such that
\begin{equation*}
\big\|f\big\|_{L^p(w)}:=\bigg(\int_{\mathbb R^n}|f(x)|^pw(x)\,dx\bigg)^{1/p}<\infty.
\end{equation*}
We also denote by $WL^p(w)$($1\leq p<\infty$) the weighted weak Lebesgue space consisting of all measurable functions $f$ such that
\begin{equation*}
\big\|f\big\|_{WL^p(w)}:=
\sup_{\lambda>0}\lambda\cdot\Big[w\big(\big\{x\in\mathbb R^n:|f(x)|>\lambda\big\}\big)\Big]^{1/p}<\infty.
\end{equation*}

We next recall some definitions and basic facts about Orlicz spaces needed for the proofs of the main results. For further information on this subject, we refer to \cite{rao}. A function $\mathcal A:[0,+\infty)\rightarrow[0,+\infty)$ is said to be a Young function if it is continuous, convex and strictly increasing satisfying $\mathcal A(0)=0$ and $\mathcal A(t)\to +\infty$ as $t\to +\infty$. An important example of Young function is $\mathcal A(t)=t^p(1+\log^+t)^p$ with some $1\leq p<\infty$. Given a Young function $\mathcal A$, we define the $\mathcal A$-average of a function $f$ over a ball $B$ by means of the following Luxemburg norm:
\begin{equation*}
\big\|f\big\|_{\mathcal A,B}
:=\inf\left\{\lambda>0:\frac{1}{|B|}\int_B\mathcal A\left(\frac{|f(x)|}{\lambda}\right)dx\leq1\right\}.
\end{equation*}
In particular, when $\mathcal A(t)=t^p$, $1\leq p<\infty$, it is easy to see that $\mathcal A$ is a Young function and
\begin{equation*}
\big\|f\big\|_{\mathcal A,B}=\left(\frac{1}{|B|}\int_B|f(x)|^p\,dx\right)^{1/p};
\end{equation*}
that is, the Luxemburg norm coincides with the normalized $L^p$ norm.Recall that the following generalization of H\"older's inequality holds:
\begin{equation*}
\frac{1}{|B|}\int_B\big|f(x)\cdot g(x)\big|\,dx\leq 2\big\|f\big\|_{\mathcal A,B}\big\|g\big\|_{\bar{\mathcal A},B},
\end{equation*}
where $\bar{\mathcal A}$ is the complementary Young function associated to $\mathcal A$, which is given by $\bar{\mathcal A}(s):=\sup_{0\leq t<\infty}[st-\mathcal A(t)]$, $0\leq s<\infty$. Obviously, $\Phi(t)=t\cdot(1+\log^+t)$ is a Young function and its complementary Young function is $\bar{\Phi}(t)\approx e^t-1$. In the present situation, we denote $\|f\|_{\Phi,B}$ and $\|g\|_{\bar{\Phi},B}$ by $\|f\|_{L\log L,B}$ and $\|g\|_{\exp L,B}$, respectively. So we have
\begin{equation}\label{holder}
\frac{1}{|B|}\int_B\big|f(x)\cdot g(x)\big|\,dx\leq 2\big\|f\big\|_{L\log L,B}\big\|g\big\|_{\exp L,B}.
\end{equation}
There is a further generalization of H\"older's inequality that turns out to be useful for our purpose (see \cite{neil}): Let $\mathcal A$, $\mathcal B$ and $\mathcal C$ be Young functions such that for all $t>0$,
\begin{equation*}
\mathcal A^{-1}(t)\cdot\mathcal B^{-1}(t)\leq\mathcal C^{-1}(t),
\end{equation*}
where $\mathcal A^{-1}(t)$ is the inverse function of $\mathcal A(t)$. Then for all functions $f$ and $g$ and all balls $B\subset\mathbb R^n$,
\begin{equation}\label{three}
\big\|f\cdot g\big\|_{\mathcal C,B}\leq 2\big\|f\big\|_{\mathcal A,B}\big\|g\big\|_{\mathcal B,B}.
\end{equation}

Let us now recall the definition of the space of $BMO(\mathbb R^n)$(see \cite{john}). $BMO(\mathbb R^n)$ is the Banach function space modulo constants with the norm $\|\cdot\|_*$ defined by
\begin{equation*}
\|b\|_*:=\sup_{B}\frac{1}{|B|}\int_B|b(x)-b_B|\,dx<\infty,
\end{equation*}
where the supremum is taken over all balls $B$ in $\mathbb R^n$ and $b_B$ stands for the mean value of $b$ over $B$; that is,
\begin{equation*}
b_B:=\frac{1}{|B|}\int_B b(y)\,dy.
\end{equation*}

\subsection{Morrey type spaces}

Let us begin with the definitions of the weighted Morrey space with two weights and generalized Morrey space.

\begin{defn}[\cite{komori}]
Let $1\leq p<\infty$ and $0<\kappa<1$. For two weights $u$ and $v$ on $\mathbb R^n$, the weighted Morrey space $\mathcal L^{p,\kappa}(v,u)$ is defined by
\begin{equation*}
\mathcal L^{p,\kappa}(v,u):=\left\{f\in L^p_{loc}(v):\big\|f\big\|_{\mathcal L^{p,\kappa}(v,u)}<\infty\right\},
\end{equation*}
where
\begin{equation}\label{WMorrey}
\big\|f\big\|_{\mathcal L^{p,\kappa}(v,u)}:=\sup_B\left(\frac{1}{u(B)^{\kappa}}\int_B|f(x)|^pv(x)\,dx\right)^{1/p}
\end{equation}
and the supremum is taken over all balls $B$ in $\mathbb R^n$. If $v=u$, then we denote $\mathcal L^{p,\kappa}(v)$, for short.
\end{defn}

\begin{defn}
Let $1\leq p<\infty$, $0<\kappa<1$ and $w$ be a weight on $\mathbb R^n$. We denote by $W\mathcal L^{p,\kappa}(w)$ the weighted weak Morrey space of all measurable functions $f$ for which
\begin{equation}\label{WWMorrey}
\big\|f\big\|_{W\mathcal L^{p,\kappa}(w)}:=\sup_B\sup_{\sigma>0}\frac{1}{w(B)^{{\kappa}/p}}\sigma
\cdot\Big[w\big(\big\{x\in B:|f(x)|>\sigma\big\}\big)\Big]^{1/p}<\infty.
\end{equation}
\end{defn}

Let $\Theta=\Theta(r)$, $r>0$, be a growth function; that is, a positive increasing function on $(0,+\infty)$ and satisfy the following doubling condition:
\begin{equation}\label{doubling}
\Theta(2r)\leq D\cdot\Theta(r), \qquad \mbox{for all }\,r>0,
\end{equation}
where $D=D(\Theta)\ge1$ is a doubling constant independent of $r$.

\begin{defn}[\cite{mizuhara}]
Let $1\leq p<\infty$ and $\Theta$ be a growth function on $(0,+\infty)$. Then the generalized Morrey space $\mathcal L^{p,\Theta}(\mathbb R^n)$ is defined by
\begin{equation*}
\mathcal L^{p,\Theta}(\mathbb R^n):=\Big\{f\in L^p_{loc}(\mathbb R^n):\big\|f\big\|_{\mathcal L^{p,\Theta}(\mathbb R^n)}<\infty\Big\},
\end{equation*}
where
\begin{equation*}
\big\|f\big\|_{\mathcal L^{p,\Theta}(\mathbb R^n)}
:=\sup_{r>0;B(x_0,r)}\bigg(\frac{1}{\Theta(r)}\int_{B(x_0,r)}|f(x)|^p\,dx\bigg)^{1/p}
\end{equation*}
and the supremum is taken over all balls $B(x_0,r)$ in $\mathbb R^n$ with $x_0\in\mathbb R^n$.
\end{defn}

\begin{defn}
Let $1\leq p<\infty$ and $\Theta$ be a growth function on $(0,+\infty)$. We denote by $W\mathcal L^{p,\Theta}(\mathbb R^n)$ the generalized weak Morrey space of all measurable functions $f$ for which
\begin{equation*}
\big\|f\big\|_{W\mathcal L^{p,\Theta}(\mathbb R^n)}:=\sup_{B(x_0,r)}\sup_{\lambda>0}\frac{1}{\Theta(r)^{1/p}}\lambda\cdot\big|\big\{x\in B(x_0,r):|f(x)|>\lambda\big\}\big|^{1/p}<\infty.
\end{equation*}
\end{defn}

In order to unify the definitions given above, we now introduce Morrey type spaces associated to $\theta$ as follows.
Let $0\leq\kappa<1$. Assume that $\theta(\cdot)$ is a positive increasing function defined in $(0,+\infty)$ and satisfies the following $\mathcal D_\kappa$ condition:
\begin{equation}\label{D condition}
\frac{\theta(\xi)}{\xi^\kappa}\leq C\cdot\frac{\theta(\xi')}{(\xi')^\kappa},
\qquad \mbox{for any}\;0<\xi'<\xi<+\infty,
\end{equation}
where $C>0$ is a constant independent of $\xi$ and $\xi'$.

\begin{defn}
Let $1\leq p<\infty$, $0\leq\kappa<1$ and $\theta$ satisfy the $\mathcal D_\kappa$ condition $(\ref{D condition})$. For two weights $u$ and $v$ on $\mathbb R^n$, we denote by $\mathcal M^{p,\theta}(v,u)$ the generalized weighted Morrey space, the space of all locally integrable functions $f$ with finite norm.
\begin{equation*}
\mathcal M^{p,\theta}(v,u):=\Big\{f\in L^p_{loc}(v):\big\|f\big\|_{\mathcal M^{p,\theta}(v,u)}<\infty\Big\},
\end{equation*}
where the norm is given by
\begin{equation*}
\big\|f\big\|_{\mathcal M^{p,\theta}(v,u)}
:=\sup_B\left(\frac{1}{\theta(u(B))}\int_B |f(x)|^pv(x)\,dx\right)^{1/p}.
\end{equation*}
Here the supremum is taken over all balls $B$ in $\mathbb R^n$. If $v=u$, then we denote $\mathcal M^{p,\theta}(v)$, for short. Furthermore, we denote by $W\mathcal M^{p,\theta}(v)$ the generalized weighted weak Morrey space of all measurable functions $f$ for which
\begin{equation*}
\big\|f\big\|_{W\mathcal M^{p,\theta}(v)}:=\sup_B\sup_{\sigma>0}\frac{1}{\theta(v(B))^{1/p}}\sigma
\cdot\Big[v\big(\big\{x\in B:|f(x)|>\sigma\big\}\big)\Big]^{1/p}<\infty.
\end{equation*}
\end{defn}

According to this definition, we recover the spaces $\mathcal L^{p,\kappa}(v,u)$ and $W\mathcal L^{p,\kappa}(v)$ under the choice of $\theta(x)=x^{\kappa}$ with $0<\kappa<1$:
\begin{equation*}
 \mathcal L^{p,\kappa}(v,u)=\mathcal M^{p,\theta}(v,u)\big|_{\theta(x)=x^{\kappa}},\qquad
W\mathcal L^{p,\kappa}(v)= W\mathcal M^{p,\theta}(v)\big|_{\theta(x)=x^{\kappa}}.
\end{equation*}
Also, note that if $\theta(x)\equiv 1$, then $\mathcal M^{p,\theta}(v)=L^p(v)$ and $W\mathcal M^{p,\theta}(v)=WL^p(v)$, the classical weighted Lebesgue and weak Lebesgue spaces.

The aim of this paper is to extend Theorems 1.1--1.4 to the corresponding Morrey type spaces. Our main results on the boundedness of $I_{\alpha}$ in the Morrey type spaces associated to $\theta$ can be formulated as follows.

\begin{theorem}\label{mainthm:1}
Let $0<\alpha<n$, $1<p<n/{\alpha}$, $1/q=1/p-{\alpha}/n$ and $w\in A_{p,q}$. Assume that $\theta$ satisfies the $\mathcal D_\kappa$ condition $(\ref{D condition})$ with $0\leq\kappa<p/q$, then the fractional integral operator $I_{\alpha}$ is bounded from $\mathcal M^{p,\theta}(w^p,w^q)$ into $\mathcal M^{q,\theta^{q/p}}(w^q)$.
\end{theorem}

\begin{theorem}\label{mainthm:2}
Let $0<\alpha<n$, $p=1$, $q=n/{(n-\alpha)}$ and $w\in A_{1,q}$. Assume that $\theta$ satisfies the $\mathcal D_\kappa$ condition $(\ref{D condition})$ with $0\leq\kappa<1/q$, then the fractional integral operator $I_{\alpha}$ is bounded from $\mathcal M^{1,\theta}(w,w^q)$ into $W\mathcal M^{q,\theta^q}(w^q)$.
\end{theorem}

Let $[b,I_{\alpha}]$ be the commutator formed by $I_{\alpha}$ and BMO function $b$. For the strong type estimate of the linear commutator $[b,I_{\alpha}]$ in the Morrey type spaces associated to $\theta$, we will prove
\begin{theorem}\label{mainthm:3}
Let $0<\alpha<n$, $1<p<n/{\alpha}$, $1/q=1/p-{\alpha}/n$ and $w\in A_{p,q}$. Assume that $\theta$ satisfies the $\mathcal D_\kappa$ condition $(\ref{D condition})$ with $0\leq\kappa<p/q$ and $b\in BMO(\mathbb R^n)$, then the commutator operator $[b,I_{\alpha}]$ is bounded from $\mathcal M^{p,\theta}(w^p,w^q)$ into $\mathcal M^{q,\theta^{q/p}}(w^q)$.
\end{theorem}

To obtain endpoint estimate for the linear commutator $[b,I_{\alpha}]$, we first need to define the weighted $\mathcal A$-average of a function $f$ over a ball $B$ by means of the weighted Luxemburg norm; that is, given a Young function $\mathcal A$ and $w\in A_\infty$, we define (see \cite{rao,zhang} for instance)
\begin{equation*}
\big\|f\big\|_{\mathcal A(w),B}:=\inf\left\{\sigma>0:\frac{1}{w(B)}
\int_B\mathcal A\left(\frac{|f(x)|}{\sigma}\right)\cdot w(x)\,dx\leq1\right\}.
\end{equation*}
When $\mathcal A(t)=t$, this norm is denoted by $\|\cdot\|_{L(w),B}$, and when $\Phi(t)=t\cdot(1+\log^+t)$, this norm is also denoted by $\|\cdot\|_{L\log L(w),B}$. The complementary Young function of $\Phi(t)$ is $\bar{\Phi}(t)\approx e^t-1$ with mean Luxemburg norm denoted by $\|\cdot\|_{\exp L(w),B}$. For $w\in A_\infty$ and for every ball $B$ in $\mathbb R^n$, we can also show the weighted version of \eqref{holder}. Namely, the following generalized H\"older's inequality in the weighted setting
\begin{equation}\label{Wholder}
\frac{1}{w(B)}\int_B|f(x)\cdot g(x)|w(x)\,dx\leq C\big\|f\big\|_{L\log L(w),B}\big\|g\big\|_{\exp L(w),B}
\end{equation}
is true (see \cite{zhang} for instance). Now we introduce new Morrey type spaces of $L\log L$ type associated to $\theta$ as follows.

\begin{defn}\label{logL}
Let $p=1$, $0\leq\kappa<1$ and $\theta$ satisfy the $\mathcal D_\kappa$ condition $(\ref{D condition})$. For two weights $u$ and $v$ on $\mathbb R^n$, we denote by $\mathcal M^{1,\theta}_{L\log L}(v,u)$ the generalized weighted Morrey space of $L\log L$ type, the space of all locally integrable functions $f$ defined on $\mathbb R^n$ with finite norm $\big\|f\big\|_{\mathcal M^{1,\theta}_{L\log L}(v,u)}$.
\begin{equation*}
\mathcal M^{1,\theta}_{L\log L}(v,u):=\left\{f\in L^1_{loc}(v):\big\|f\big\|_{\mathcal M^{1,\theta}_{L\log L}(v,u)}<\infty\right\},
\end{equation*}
where
\begin{equation*}
\big\|f\big\|_{\mathcal M^{1,\theta}_{L\log L}(v,u)}
:=\sup_B\left\{\frac{v(B)}{\theta(u(B))}\cdot\big\|f\big\|_{L\log L(v),B}\right\}.
\end{equation*}
Here the supremum is taken over all balls $B$ in $\mathbb R^n$. If $v=u$, then we denote $\mathcal M^{1,\theta}_{L\log L}(v)$ for brevity.
\end{defn}

Note that $t\leq t\cdot(1+\log^+t)$ for all $t>0$, then for any ball $B\subset\mathbb R^n$ and $v\in A_\infty$, we have $\big\|f\big\|_{L(v),B}\leq \big\|f\big\|_{L\log L(v),B}$ by definition, i.e., the inequality
\begin{equation}\label{main esti1}
\big\|f\big\|_{L(v),B}=\frac{1}{v(B)}\int_B|f(x)|\cdot v(x)\,dx\leq\big\|f\big\|_{L\log L(v),B}
\end{equation}
holds for any ball $B\subset\mathbb R^n$. From this, we can further see that when $\theta$ satisfies the $\mathcal D_\kappa$ condition $(\ref{D condition})$ with $0\leq\kappa<1$, and $u$ is another weight function,
\begin{equation}\label{main esti2}
\begin{split}
\frac{1}{\theta(u(B))}\int_B|f(x)|\cdot v(x)\,dx
&=\frac{v(B)}{\theta(u(B))}\cdot\frac{1}{v(B)}\int_B|f(x)|\cdot v(x)\,dx\\
&=\frac{v(B)}{\theta(u(B))}\cdot\big\|f\big\|_{L(v),B}\\
&\leq\frac{v(B)}{\theta(u(B))}\cdot\big\|f\big\|_{L\log L(v),B}.
\end{split}
\end{equation}
Hence, we have $\mathcal M^{1,\theta}_{L\log L}(v,u)\subset\mathcal M^{1,\theta}(v,u)$ by definition.

In Definition \ref{logL}, we also consider the special case when $\theta$ is taken to be $\theta(x)=x^\kappa$ with $0<\kappa<1$, and denote the corresponding space by $\mathcal L^{1,\kappa}_{L\log L}(v,u)$.

\begin{defn}
Let $p=1$ and $0<\kappa<1$. For two weights $u$ and $v$ on $\mathbb R^n$, we denote by $\mathcal L^{1,\kappa}_{L\log L}(v,u)$ the weighted Morrey space of $L\log L$ type, the space of all locally integrable functions $f$ defined on $\mathbb R^n$ with finite norm $\big\|f\big\|_{\mathcal L^{1,\kappa}_{L\log L}(v,u)}$.
\begin{equation*}
\mathcal L^{1,\kappa}_{L\log L}(v,u):=\left\{f\in L^1_{loc}(v):\big\|f\big\|_{\mathcal L^{1,\kappa}_{L\log L}(v,u)}<\infty\right\},
\end{equation*}
where
\begin{equation*}
\big\|f\big\|_{\mathcal L^{1,\kappa}_{L\log L}(v,u)}
:=\sup_B\left\{\frac{v(B)}{u(B)^{\kappa}}\cdot\big\|f\big\|_{L\log L(v),B}\right\}.
\end{equation*}
In this situation, we have $\mathcal L^{1,\kappa}_{L\log L}(v,u)\subset\mathcal L^{1,\kappa}(v,u)$.
\end{defn}

In the endpoint case $p=1$, we will prove the following weak type $L\log L$ estimate of the linear commutator $[b,I_{\alpha}]$ in the Morrey type space associated to $\theta$.

\begin{theorem}\label{mainthm:4}
Let $0<\alpha<n$, $p=1$, $q=n/{(n-\alpha)}$ and $w\in A_{1,q}$. Assume that $\theta$ satisfies the $\mathcal D_\kappa$ condition $(\ref{D condition})$ with $0\leq\kappa<1/q$ and $b\in BMO(\mathbb R^n)$, then for any given $\sigma>0$ and any ball $B\subset\mathbb R^n$, there exists a constant $C>0$ independent of $f$, $B$ and $\sigma>0$ such that
\begin{equation*}
\begin{split}
&\frac{1}{\theta(w^q(B))}\Big[w^q\big(\big\{x\in B:\big|[b,I_{\alpha}](f)(x)\big|>\sigma\big\}\big)\Big]^{1/q}
\leq C\cdot\bigg\|\Phi\left(\frac{|f|}{\,\sigma\,}\right)\bigg\|_{\mathcal M^{1,\theta}_{L\log L}(w,w^q)},
\end{split}
\end{equation*}
where $\Phi(t)=t\cdot(1+\log^+t)$. From the definitions, we can roughly say that the commutator operator $[b,I_{\alpha}]$ is bounded from $\mathcal M^{1,\theta}_{L\log L}(w,w^q)$ into $W\mathcal M^{q,\theta^q}(w^q)$.
\end{theorem}

In particular, if we take $\theta(x)=x^\kappa$ with $0<\kappa<1$, then we immediately
get the following strong type estimate and endpoint estimate of $I_{\alpha}$ and $[b,I_{\alpha}]$ in the weighted Morrey spaces.

\begin{corollary}\label{cor:1}
Let $0<\alpha<n$, $1<p<n/{\alpha}$, $1/q=1/p-{\alpha}/n$ and $w\in A_{p,q}$. If $0<\kappa<p/q$, then the fractional integral operator $I_{\alpha}$ is bounded from $\mathcal L^{p,\kappa}(w^p,w^q)$ into $\mathcal L^{q,{\kappa q}/p}(w^q)$.
\end{corollary}

\begin{corollary}\label{cor:2}
Let $0<\alpha<n$, $p=1$, $q=n/{(n-\alpha)}$ and $w\in A_{1,q}$. If $0<\kappa<1/q$, then the fractional integral operator $I_{\alpha}$ is bounded from $\mathcal L^{1,\kappa}(w,w^q)$ into $W\mathcal L^{q,\kappa q}(w^q)$.
\end{corollary}

\begin{corollary}\label{cor:3}
Let $0<\alpha<n$, $1<p<n/{\alpha}$, $1/q=1/p-{\alpha}/n$ and $w\in A_{p,q}$. If $0<\kappa<p/q$ and $b\in BMO(\mathbb R^n)$, then the commutator operator $[b,I_{\alpha}]$ is bounded from $\mathcal L^{p,\kappa}(w^p,w^q)$ into $\mathcal L^{q,{\kappa q}/p}(w^q)$.
\end{corollary}

\begin{corollary}\label{cor:4}
Let $0<\alpha<n$, $p=1$, $q=n/{(n-\alpha)}$ and $w\in A_{1,q}$. If $0<\kappa<1/q$ and $b\in BMO(\mathbb R^n)$, then for any given $\sigma>0$ and any ball $B\subset\mathbb R^n$, there exists a constant $C>0$ independent of $f$, $B$ and $\sigma>0$ such that
\begin{equation*}
\begin{split}
&\frac{1}{w^q(B)^{\kappa}}\Big[w^q\big(\big\{x\in B:\big|[b,I_{\alpha}](f)(x)\big|>\sigma\big\}\big)\Big]^{1/q}
\leq C\cdot\bigg\|\Phi\left(\frac{|f|}{\,\sigma\,}\right)\bigg\|_{\mathcal L^{1,\kappa}_{L\log L}(w,w^q)},
\end{split}
\end{equation*}
where $\Phi(t)=t\cdot(1+\log^+t)$.
\end{corollary}

Moreover, for the extreme case $\kappa=p/q$ of Corollary \ref{cor:1}, we will show that $I_{\alpha}$ is bounded from $\mathcal L^{p,\kappa}(w^p,w^q)$ into $BMO(\mathbb R^n)$.
\begin{theorem}\label{mainthm:end1}
Let $0<\alpha<n$, $1<p<n/{\alpha}$, $1/q=1/p-{\alpha}/n$ and $w\in A_{p,q}$. If $\kappa=p/q$, then the fractional integral operator $I_{\alpha}$ is bounded from $\mathcal L^{p,\kappa}(w^p,w^q)$ into $BMO(\mathbb R^n)$.
\end{theorem}

It should be pointed out that Corollaries \ref{cor:1} through \ref{cor:3} were given by Komori and Shirai in \cite{komori}. Corollary \ref{cor:4} and Theorem \ref{mainthm:end1} are new results.

\begin{defn}
In the unweighted case (when $u=v\equiv1$), we denote the corresponding unweighted Morrey type spaces associated to $\theta$ by $\mathcal M^{p,\theta}(\mathbb R^n)$, $W\mathcal M^{p,\theta}(\mathbb R^n)$ and $\mathcal M^{1,\theta}_{L\log L}(\mathbb R^n)$, respectively. That is, let $1\leq p<\infty$ and $\theta$ satisfy the $\mathcal D_\kappa$ condition $(\ref{D condition})$ with $0\leq\kappa<1$, we define
\begin{equation*}
\mathcal M^{p,\theta}(\mathbb R^n):=\left\{f\in L^p_{loc}(\mathbb R^n):
\big\|f\big\|_{\mathcal M^{p,\theta}(\mathbb R^n)}=\sup_B\bigg(\frac{1}{\theta(|B|)}\int_B |f(x)|^p\,dx\bigg)^{1/p}<\infty\right\},
\end{equation*}
\begin{equation*}
W\mathcal M^{p,\theta}(\mathbb R^n):=\left\{f:
\big\|f\big\|_{W\mathcal M^{p,\theta}(\mathbb R^n)}=\sup_B\sup_{\sigma>0}\frac{1}{\theta(|B|)^{1/p}}\sigma
\cdot \Big|\big\{x\in B:|f(x)|>\sigma\big\}\Big|^{1/p}<\infty\right\},
\end{equation*}
and
\begin{equation*}
\mathcal M^{1,\theta}_{L\log L}(\mathbb R^n):=\left\{f\in L^1_{loc}(\mathbb R^n):
\big\|f\big\|_{\mathcal M^{1,\theta}_{L\log L}(\mathbb R^n)}
=\sup_B\left(\frac{|B|}{\theta(|B|)}\cdot\big\|f\big\|_{L\log L,B}\right)<\infty\right\}.
\end{equation*}
\end{defn}

Naturally, when $u(x)=v(x)\equiv 1$ we have the following unweighted results.
\begin{corollary}\label{cor:5}
Let $0<\alpha<n$, $1<p<n/{\alpha}$ and $1/q=1/p-{\alpha}/n$. Assume that $\theta$ satisfies the $\mathcal D_\kappa$ condition $(\ref{D condition})$ with $0\leq\kappa<p/q$, then the fractional integral operator $I_{\alpha}$ is bounded from $\mathcal M^{p,\theta}(\mathbb R^n)$ into $\mathcal M^{q,\theta^{q/p}}(\mathbb R^n)$.
\end{corollary}

\begin{corollary}\label{cor:6}
Let $0<\alpha<n$, $p=1$ and $q=n/{(n-\alpha)}$. Assume that $\theta$ satisfies the $\mathcal D_\kappa$ condition $(\ref{D condition})$ with $0\leq\kappa<1/q$, then the fractional integral operator $I_{\alpha}$ is bounded from $\mathcal M^{1,\theta}(\mathbb R^n)$ into $W\mathcal M^{q,\theta^q}(\mathbb R^n)$.
\end{corollary}

\begin{corollary}\label{cor:7}
Let $0<\alpha<n$, $1<p<n/{\alpha}$ and $1/q=1/p-{\alpha}/n$. Assume that $\theta$ satisfies the $\mathcal D_\kappa$ condition $(\ref{D condition})$ with $0\leq\kappa< p/q$ and $b\in BMO(\mathbb R^n)$, then the commutator operator $[b,I_{\alpha}]$ is bounded from $\mathcal M^{p,\theta}(\mathbb R^n)$ into $\mathcal M^{q,\theta^{q/p}}(\mathbb R^n)$.
\end{corollary}

\begin{corollary}\label{cor:8}
Let $0<\alpha<n$, $p=1$ and $q=n/{(n-\alpha)}$. Assume that $\theta$ satisfies the $\mathcal D_\kappa$ condition $(\ref{D condition})$ with $0\leq\kappa<1/q$ and $b\in BMO(\mathbb R^n)$, then for any given $\sigma>0$ and any ball $B\subset\mathbb R^n$, there exists a constant $C>0$ independent of $f$, $B$ and $\sigma>0$ such that
\begin{equation*}
\begin{split}
&\frac{1}{\theta(|B|)}\Big|\big\{x\in B:\big|[b,I_{\alpha}](f)(x)\big|>\sigma\big\}\Big|^{1/q}
\leq C\cdot\bigg\|\Phi\left(\frac{|f|}{\,\sigma\,}\right)\bigg\|_{\mathcal M^{1,\theta}_{L\log L}(\mathbb R^n)},
\end{split}
\end{equation*}
where $\Phi(t)=t\cdot(1+\log^+t)$.
\end{corollary}

We also introduce the generalized Morrey space of $L\log L$ type.

\begin{defn}
Let $p=1$ and $\Theta$ be a growth function on $(0,+\infty)$. We denote by $\mathcal L^{1,\Theta}_{L\log L}(\mathbb R^n)$ the generalized Morrey space of $L\log L$ type, which is given by
\begin{equation*}
\mathcal L^{1,\Theta}_{L\log L}(\mathbb R^n):=\left\{f\in L^1_{loc}(\mathbb R^n):
\big\|f\big\|_{\mathcal L^{1,\Theta}_{L\log L}(\mathbb R^n)}<\infty\right\},
\end{equation*}
where
\begin{equation*}
\big\|f\big\|_{\mathcal L^{1,\Theta}_{L\log L}(\mathbb R^n)}
:=\sup_{r>0;B(x_0,r)}\left\{\frac{|B(x_0,r)|}{\Theta(r)}\cdot\big\|f\big\|_{L\log L,B(x_0,r)}\right\}.
\end{equation*}
In this situation, we also have $\mathcal L^{1,\Theta}_{L\log L}(\mathbb R^n)\subset\mathcal L^{1,\Theta}(\mathbb R^n)$.
\end{defn}

Below we are going to show that our new Morrey type spaces can be reduced to generalized Morrey spaces. In fact, assume that $\theta(\cdot)$ is a positive increasing function defined in $(0,+\infty)$ and satisfies the $\mathcal D_\kappa$ condition $(\ref{D condition})$ with some $0\leq\kappa<1$. For any fixed $x_0\in\mathbb R^n$ and $r>0$, we set $\Theta(r):=\theta(|B(x_0,r)|)$. Observe that
\begin{equation*}
\Theta(2r)=\theta\big(|B(x_0,2r)|\big)=\theta\big(2^n|B(x_0,r)|\big).
\end{equation*}
Then it is easy to verify that $\Theta(r)$, $r>0$, is a growth function with doubling constant $D(\Theta):1\le D(\Theta)<2^n$. Hence, by the choice of $\Theta$ mentioned above, we get $\mathcal M^{p,\theta}(\mathbb R^n)=\mathcal L^{p,\Theta}(\mathbb R^n)$ and $W\mathcal M^{p,\theta}(\mathbb R^n)=W\mathcal L^{p,\Theta}(\mathbb R^n)$ for $p\in[1,+\infty)$, and $\mathcal M^{1,\theta}_{L\log L}(\mathbb R^n)=\mathcal L^{1,\Theta}_{L\log L}(\mathbb R^n)$.
Therefore, by the above unweighted results (Corollaries \ref{cor:5}--\ref{cor:8}), we can also obtain strong type estimate and endpoint estimate of $I_{\alpha}$ and $[b,I_{\alpha}]$ in the generalized Morrey spaces.

\begin{corollary}\label{cor:9}
Let $0<\alpha<n$, $1<p<n/{\alpha}$ and $1/q=1/p-{\alpha}/n$. Suppose that $\Theta$ satisfies the doubling condition $(\ref{doubling})$ and $1\le D(\Theta)<2^{{np}/q}$, then the fractional integral operator $I_{\alpha}$ is bounded from $\mathcal L^{p,\Theta}(\mathbb R^n)$ into $\mathcal L^{q,\Theta^{q/p}}(\mathbb R^n)$.
\end{corollary}

\begin{corollary}\label{cor:10}
Let $0<\alpha<n$, $p=1$ and $q=n/{(n-\alpha)}$. Suppose that $\Theta$ satisfies the doubling condition $(\ref{doubling})$ and $1\le D(\Theta)<2^{n/q}$, then the fractional integral operator $I_{\alpha}$ is bounded from $\mathcal L^{1,\Theta}(\mathbb R^n)$ into $W\mathcal L^{q,\Theta^{q}}(\mathbb R^n)$.
\end{corollary}

\begin{corollary}\label{cor:11}
Let $0<\alpha<n$, $1<p<n/{\alpha}$ and $1/q=1/p-{\alpha}/n$. Suppose that $\Theta$ satisfies the doubling condition $(\ref{doubling})$ with $1\le D(\Theta)<2^{{np}/q}$ and $b\in BMO(\mathbb R^n)$, then the commutator operator $[b,I_{\alpha}]$ is bounded from $\mathcal L^{p,\Theta}(\mathbb R^n)$ into $\mathcal L^{q,\Theta^{q/p}}(\mathbb R^n)$.
\end{corollary}

\begin{corollary}\label{cor:12}
Let $0<\alpha<n$, $p=1$ and $q=n/{(n-\alpha)}$. Suppose that $\Theta$ satisfies the doubling condition $(\ref{doubling})$ with $1\le D(\Theta)<2^{n/q}$ and $b\in BMO(\mathbb R^n)$, then for any given $\sigma>0$ and any ball $B(x_0,r)\subset\mathbb R^n$, there exists a constant $C>0$ independent of $f$, $B(x_0,r)$ and $\sigma>0$ such that
\begin{equation*}
\begin{split}
&\frac{1}{\Theta(r)}\Big|\big\{x\in B(x_0,r):\big|[b,I_{\alpha}](f)(x)\big|>\sigma\big\}\Big|^{1/q}
\leq C\cdot\bigg\|\Phi\left(\frac{|f|}{\,\sigma\,}\right)\bigg\|_{\mathcal L^{1,\Theta}_{L\log L}(\mathbb R^n)},
\end{split}
\end{equation*}
where $\Phi(t)=t\cdot(1+\log^+t)$.
\end{corollary}

We will also prove the following result which can be regarded as a supplement of Corollaries \ref{cor:9} and \ref{cor:10}.

\begin{theorem}\label{mainthm:end2}
Let $0<\alpha<n$, $1<p<n/{\alpha}$ and $1/q=1/p-{\alpha}/n$. Suppose that $\Theta$ satisfies the following condition:
\begin{equation}\label{assump}
\Theta(r)\leq C\cdot r^{{np}/q}, \qquad \mbox{for all }\,r>0,
\end{equation}
where $C=C(\Theta)>0$ is a universal constant independent of $r$. Then the fractional integral operator $I_{\alpha}$ is bounded from $\mathcal L^{p,\Theta}(\mathbb R^n)$ into $BMO(\mathbb R^n)$.
\end{theorem}

It is worth pointing out that Corollaries \ref{cor:9} through \ref{cor:11} were obtained by Nakai in \cite{nakai}. Corollary \ref{cor:12} and Theorem \ref{mainthm:end2} seem to be new, as far as we know.

Throughout this paper, the letter $C$ always denotes a positive constant that is independent of the essential variables but whose value may vary at each occurrence. We also use $A\approx B$ to denote the equivalence of $A$ and $B$; that is, there exist two positive constants $C_1$, $C_2$ independent of quantities $A$ and $B$ such that $C_1 A\leq B\leq C_2 A$. Equivalently, we could define the above notions of this section with cubes in place of balls and we will use whichever is more appropriate, depending on the circumstances.

\section{Proofs of Theorems \ref{mainthm:1} and \ref{mainthm:2}}

\begin{proof}[Proof of Theorem $\ref{mainthm:1}$]
Here and in what follows, for any positive number $\gamma>0$, we denote $f^\gamma(x):=[f(x)]^{\gamma}$ by convention. For example, when $1<p<q<\infty$, we have $[f^{q/p}(x)]^{1/q}=[f(x)]^{1/p}$. Let $f\in\mathcal M^{p,\theta}(w^p,w^q)$ with $1<p,q<\infty$ and $w\in A_{p,q}$. For an arbitrary point $x_0\in\mathbb R^n$, set $B=B(x_0,r_B)$ for the ball centered at $x_0$ and of radius $r_B$, $2B=B(x_0,2r_B)$. We represent $f$ as
\begin{equation*}
f=f\cdot\chi_{2B}+f\cdot\chi_{(2B)^c}:=f_1+f_2;
\end{equation*}
by the linearity of the fractional integral operator $I_{\alpha}$, one can write
\begin{equation*}
\begin{split}
&\frac{1}{\theta(w^q(B))^{1/p}}\bigg(\int_B\big|I_{\alpha}(f)(x)\big|^qw^q(x)\,dx\bigg)^{1/q}\\
\leq\ &\frac{1}{\theta(w^q(B))^{1/p}}\bigg(\int_B\big|I_{\alpha}(f_1)(x)\big|^qw^q(x)\,dx\bigg)^{1/q}\\
&+\frac{1}{\theta(w^q(B))^{1/p}}\bigg(\int_B\big|I_{\alpha}(f_2)(x)\big|^qw^q(x)\,dx\bigg)^{1/q}\\
:=\ &I_1+I_2.
\end{split}
\end{equation*}
Below we will give the estimates of $I_1$ and $I_2$, respectively. By the weighted $(L^p,L^q)$-boundedness of $I_{\alpha}$ (see Theorem \ref{strong}), we have
\begin{equation*}
\begin{split}
I_1&\leq \frac{1}{\theta(w^q(B))^{1/p}}
\big\|I_{\alpha}(f_1)\big\|_{L^q(w^q)}\\
&\leq C\cdot\frac{1}{\theta(w^q(B))^{1/p}}
\bigg(\int_{2B}|f(x)|^p w^p(x)\,dx\bigg)^{1/p}\\
&\leq C\big\|f\big\|_{\mathcal M^{p,\theta}(w^p,w^q)}
\cdot\frac{\theta(w^q(2B))^{1/p}}{\theta(w^q(B))^{1/p}}.
\end{split}
\end{equation*}
Since $w\in A_{p,q}$, we get $w^q\in A_q\subset A_{\infty}$ by Lemma \ref{relation}$(i)$. Moreover, since $0<w^q(B)<w^q(2B)<+\infty$ when $w^q\in A_q$ with $1<q<\infty$, then by the $\mathcal D_\kappa$ condition (\ref{D condition}) of $\theta$ and inequality (\ref{weights}), we obtain
\begin{equation*}
\begin{split}
I_1&\leq C\big\|f\big\|_{\mathcal M^{p,\theta}(w^p,w^q)}\cdot\frac{w^q(2B)^{\kappa/p}}{w^q(B)^{\kappa/p}}\\
&\leq C\big\|f\big\|_{\mathcal M^{p,\theta}(w^p,w^q)}.
\end{split}
\end{equation*}
As for the term $I_2$, it is clear that when $x\in B$ and $y\in(2B)^c$, we get $|x-y|\approx|x_0-y|$. We then decompose $\mathbb R^n$ into a geometrically increasing sequence of concentric balls, and obtain the following pointwise estimate:
\begin{align}\label{pointwise1}
\big|I_{\alpha}(f_2)(x)\big|&\leq\int_{\mathbb R^n}\frac{|f_2(y)|}{|x-y|^{n-\alpha}}dy
\leq C\int_{(2B)^c}\frac{|f(y)|}{|x_0-y|^{n-\alpha}}dy\notag\\
&\leq C\sum_{j=1}^\infty\frac{1}{|2^{j+1}B|^{1-{\alpha}/n}}\int_{2^{j+1}B}|f(y)|\,dy.
\end{align}
From this, it follows that
\begin{equation*}
I_2\leq C\cdot\frac{w^q(B)^{1/q}}{\theta(w^q(B))^{1/p}}\sum_{j=1}^\infty
\frac{1}{|2^{j+1}B|^{1-{\alpha}/n}}\int_{2^{j+1}B}|f(y)|\,dy.
\end{equation*}
By using H\"older's inequality and $A_{p,q}$ condition on $w$, we get
\begin{equation*}
\begin{split}
&\frac{1}{|2^{j+1}B|^{1-{\alpha}/n}}\int_{2^{j+1}B}|f(y)|\,dy\\
&\leq\frac{1}{|2^{j+1}B|^{1-{\alpha}/n}}\bigg(\int_{2^{j+1}B}\big|f(y)\big|^pw^p(y)\,dy\bigg)^{1/p}
\left(\int_{2^{j+1}B}w(y)^{-p'}dy\right)^{1/{p'}}\\
&\leq C\big\|f\big\|_{\mathcal M^{p,\theta}(w^p,w^q)}\cdot\frac{\theta(w^q(2^{j+1}B))^{1/p}}{w^q(2^{j+1}B)^{1/q}}.
\end{split}
\end{equation*}
Hence
\begin{equation*}
\begin{split}
I_2&\leq C\big\|f\big\|_{\mathcal M^{p,\theta}(w^p,w^q)}
\times\sum_{j=1}^\infty\frac{\theta(w^q(2^{j+1}B))^{1/p}}{\theta(w^q(B))^{1/p}}\cdot\frac{w^q(B)^{1/q}}{w^q(2^{j+1}B)^{1/q}}.
\end{split}
\end{equation*}
Notice that $w^q\in A_q\subset A_\infty$ for $1<q<\infty$, then by using the $\mathcal D_\kappa$ condition (\ref{D condition}) of $\theta$ again, the inequality (\ref{compare}) with exponent $\delta>0$ and the fact that $0\leq\kappa<p/q$, we find that
\begin{align}\label{theta1}
\sum_{j=1}^\infty\frac{\theta(w^q(2^{j+1}B))^{1/p}}{\theta(w^q(B))^{1/p}}\cdot\frac{w^q(B)^{1/q}}{w^q(2^{j+1}B)^{1/q}}
&\leq C\sum_{j=1}^\infty\frac{w^q(B)^{1/q-{\kappa}/p}}{w^q(2^{j+1}B)^{1/q-{\kappa}/p}}\notag\\
&\leq C\sum_{j=1}^\infty\left(\frac{|B|}{|2^{j+1}B|}\right)^{\delta {(1/q-\kappa/p)}}\notag\\
&\leq C\sum_{j=1}^\infty\left(\frac{1}{2^{(j+1)n}}\right)^{\delta {(1/q-\kappa/p)}}\notag\\
&\leq C,
\end{align}
which gives our desired estimate $I_2\leq C\big\|f\big\|_{\mathcal M^{p,\theta}(w^p,w^q)}$. Combining the above estimates for $I_1$ and $I_2$, and then taking the supremum over all balls $B\subset\mathbb R^n$, we complete the proof of Theorem \ref{mainthm:1}.
\end{proof}

\begin{proof}[Proof of Theorem $\ref{mainthm:2}$]
Let $f\in\mathcal M^{1,\theta}(w,w^q)$ with $1<q<\infty$ and $w\in A_{1,q}$. For an arbitrary ball $B=B(x_0,r_B)\subset\mathbb R^n$, we represent $f$ as
\begin{equation*}
f=f\cdot\chi_{2B}+f\cdot\chi_{(2B)^c}:=f_1+f_2;
\end{equation*}
then for any given $\sigma>0$, by the linearity of the fractional integral operator $I_{\alpha}$, one can write
\begin{equation*}
\begin{split}
&\frac{1}{\theta(w^q(B))}\sigma\cdot
\Big[w^q\big(\big\{x\in B:\big|I_{\alpha}(f)(x)\big|>\sigma\big\}\big)\Big]^{1/q}\\
\end{split}
\end{equation*}
\begin{equation*}
\begin{split}
\leq &\frac{1}{\theta(w^q(B))}\sigma\cdot
\Big[w^q\big(\big\{x\in B:\big|I_{\alpha}(f_1)(x)\big|>\sigma/2\big\}\big)\Big]^{1/q}\\
&+\frac{1}{\theta(w^q(B))}\sigma\cdot
\Big[w^q\big(\big\{x\in B:\big|I_{\alpha}(f_2)(x)\big|>\sigma/2\big\}\big)\Big]^{1/q}\\
:=&I'_1+I'_2.
\end{split}
\end{equation*}
We first consider the term $I'_1$. By the weighted weak $(1,q)$-boundedness of $I_{\alpha}$ (see Theorem \ref{weak}), we have
\begin{equation*}
\begin{split}
I'_1&\leq C\cdot\frac{1}{\theta(w^q(B))}\|f_1\|_{L^1(w)}\\
&=C\cdot\frac{1}{\theta(w^q(B))}\bigg(\int_{2B}|f(x)|w(x)\,dx\bigg)\\
&\leq C\big\|f\big\|_{\mathcal M^{1,\theta}(w,w^q)}
\cdot\frac{\theta(w^q(2B))}{\theta(w^q(B))}.
\end{split}
\end{equation*}
Since $w$ is in the class $A_{1,q}$, we get $w^q\in A_1\subset A_{\infty}$ by Lemma \ref{relation}$(ii)$. Moreover, since $0<w^q(B)<w^q(2B)<+\infty$ when $w^q\in A_1$, then we apply the $\mathcal D_\kappa$ condition (\ref{D condition}) of $\theta$ and inequality (\ref{weights}) to obtain that
\begin{equation*}
\begin{split}
I'_1&\leq C\big\|f\big\|_{\mathcal M^{1,\theta}(w,w^q)}
\cdot\frac{w^q(2B)^{\kappa}}{w^q(B)^{\kappa}}\\
&\leq C\big\|f\big\|_{\mathcal M^{1,\theta}(w,w^q)}.
\end{split}
\end{equation*}
As for the term $I'_2$, it follows directly from Chebyshev's inequality and the pointwise estimate \eqref{pointwise1} that
\begin{equation*}
\begin{split}
I'_2&\leq\frac{1}{\theta(w^q(B))}\sigma\cdot\frac{\,2\,}{\sigma}\bigg(\int_B\big|I_\alpha(f_2)(x)\big|^qw^q(x)\,dx\bigg)^{1/q}\\
&\leq C\cdot\frac{w^q(B)^{1/q}}{\theta(w^q(B))}\sum_{j=1}^\infty\frac{1}{|2^{j+1}B|^{1-{\alpha}/n}}\int_{2^{j+1}B}|f(y)|\,dy.
\end{split}
\end{equation*}
Moreover, by applying H\"older's inequality and then the reverse H\"older's inequality in succession, we can show that $w^q\in A_1$ if and only if $w\in A_1\cap RH_q$ (see \cite{johnson}), where $RH_q$ denotes the reverse H\"older class. Another application of $A_1$ condition on $w$ gives that
\begin{equation*}
\begin{split}
\frac{1}{|2^{j+1}B|^{1-{\alpha}/n}}\int_{2^{j+1}B}\big|f(y)\big|\,dy
&\leq C\cdot\frac{|2^{j+1}B|^{\alpha/n}}{w(2^{j+1}B)}\cdot\underset{y\in 2^{j+1}B}{\mbox{ess\,inf}}\;w(y)\int_{2^{j+1}B}|f(y)|\,dy\\
&\leq C\cdot\frac{|2^{j+1}B|^{\alpha/n}}{w(2^{j+1}B)}\bigg(\int_{2^{j+1}B}|f(y)|w(y)\,dy\bigg)\\
&\leq C\big\|f\big\|_{\mathcal M^{1,\theta}(w,w^q)}\cdot\frac{|2^{j+1}B|^{\alpha/n}}{w(2^{j+1}B)}\cdot\theta(w^q(2^{j+1}B)).
\end{split}
\end{equation*}
In addition, note that $w\in RH_q$. We are able to verify that for any $j\in\mathbb Z^+$,
\begin{equation*}
w^q(2^{j+1}B)^{1/q}=\left(\int_{2^{j+1}B}w^q(x)\,dx\right)^{1/q}\leq C\cdot|2^{j+1}B|^{1/q-1}\cdot w(2^{j+1}B),
\end{equation*}
which is equivalent to
\begin{equation}\label{wq}
\frac{|2^{j+1}B|^{\alpha/n}}{w(2^{j+1}B)}\leq C\cdot\frac{1}{w^q(2^{j+1}B)^{1/q}}.
\end{equation}
Consequently,
\begin{equation*}
\begin{split}
I'_2&\leq C\big\|f\big\|_{\mathcal M^{1,\theta}(w,w^q)}
\times\sum_{j=1}^\infty\frac{\theta(w^q(2^{j+1}B))}{\theta(w^q(B))}\cdot\frac{w^q(B)^{1/q}}{w^q(2^{j+1}B)^{1/q}}.
\end{split}
\end{equation*}
Recall that $w^q\in A_1\subset A_\infty$, therefore, by using the $\mathcal D_\kappa$ condition (\ref{D condition}) of $\theta$ again, the inequality (\ref{compare}) with exponent $\delta^*>0$ and the fact that $0\leq\kappa<1/q$, we get
\begin{align}\label{theta2}
\sum_{j=1}^\infty\frac{\theta(w^q(2^{j+1}B))}{\theta(w^q(B))}\cdot\frac{w^q(B)^{1/q}}{w^q(2^{j+1}B)^{1/q}}
&\leq C\sum_{j=1}^\infty\frac{w^q(B)^{1/q-\kappa}}{w^q(2^{j+1}B)^{1/q-\kappa}}\notag\\
&\leq C\sum_{j=1}^\infty\left(\frac{|B|}{|2^{j+1}B|}\right)^{\delta^*(1/q-\kappa)}\notag\\
&\leq C\sum_{j=1}^\infty\left(\frac{1}{2^{(j+1)n}}\right)^{\delta^*(1/q-\kappa)}\notag\\
&\leq C,
\end{align}
which implies our desired estimate $I'_2\leq C\big\|f\big\|_{\mathcal M^{1,\theta}(w,w^q)}$. Summing up the above estimates for $I'_1$ and $I'_2$, and then taking the supremum over all balls $B\subset\mathbb R^n$ and all $\sigma>0$, we finish the proof of Theorem \ref{mainthm:2}.
\end{proof}

\section{Proofs of Theorems \ref{mainthm:3} and \ref{mainthm:4}}

To prove our main theorems in this section, we need the following lemma about $BMO$ functions.

\begin{lemma}\label{BMO}
Let $b$ be a function in $BMO(\mathbb R^n)$.

$(i)$ For every ball $B$ in $\mathbb R^n$ and for all $j\in\mathbb Z^+$, then
\begin{equation*}
\big|b_{2^{j+1}B}-b_B\big|\leq C\cdot(j+1)\|b\|_*.
\end{equation*}

$(ii)$ For $1<q<\infty$, every ball $B$ in $\mathbb R^n$ and for all $\mu\in A_{\infty}$, then
\begin{equation*}
\bigg(\int_B\big|b(x)-b_B\big|^q\mu(x)\,dx\bigg)^{1/q}\leq C\|b\|_*\cdot\mu(B)^{1/q}.
\end{equation*}
\end{lemma}
\begin{proof}
For the proof of $(i)$, we refer the reader to \cite{stein2}. For the proof of $(ii)$, we refer the reader to \cite{wang}.
\end{proof}

\begin{proof}[Proof of Theorem $\ref{mainthm:3}$]
Let $f\in\mathcal M^{p,\theta}(w^p,w^q)$ with $1<p,q<\infty$ and $w\in A_{p,q}$. For each fixed ball $B=B(x_0,r_B)\subset\mathbb R^n$, as before, we represent $f$ as $f=f_1+f_2$, where $f_1=f\cdot\chi_{2B}$, $2B=B(x_0,2r_B)\subset\mathbb R^n$. By the linearity of the commutator operator $[b,I_{\alpha}]$, we write
\begin{equation*}
\begin{split}
&\frac{1}{\theta(w^q(B))^{1/p}}\bigg(\int_B\big|[b,I_{\alpha}](f)(x)\big|^qw^q(x)\,dx\bigg)^{1/q}\\
\leq\ &\frac{1}{\theta(w^q(B))^{1/p}}\bigg(\int_B\big|[b,I_{\alpha}](f_1)(x)\big|^qw^q(x)\,dx\bigg)^{1/q}\\
&+\frac{1}{\theta(w^q(B))^{1/p}}\bigg(\int_B\big|[b,I_{\alpha}](f_2)(x)\big|^qw^q(x)\,dx\bigg)^{1/q}\\
:=\ &J_1+J_2.
\end{split}
\end{equation*}
Since $w$ is in the class $A_{p,q}$, we get $w^q\in A_q\subset A_{\infty}$ by Lemma \ref{relation}$(i)$. By using Theorem \ref{cstrong}, the $\mathcal D_\kappa$ condition (\ref{D condition}) of $\theta$ and inequality (\ref{weights}), we obtain
\begin{equation*}
\begin{split}
J_1&\leq\frac{1}{\theta(w^q(B))^{1/p}}\big\|[b,I_{\alpha}](f_1)\big\|_{L^q(w^q)}\\
&\leq C\cdot\frac{1}{\theta(w^q(B))^{1/p}}
\bigg(\int_{2B}|f(x)|^p w^p(x)\,dx\bigg)^{1/p}\\
&\leq C\big\|f\big\|_{\mathcal M^{p,\theta}(w^p,w^q)}
\cdot\frac{\theta(w^q(2B))^{1/p}}{\theta(w^q(B))^{1/p}}\\
&\leq C\big\|f\big\|_{\mathcal M^{p,\theta}(w^p,w^q)}\cdot\frac{w^q(2B)^{\kappa/p}}{w^q(B)^{\kappa/p}}\\
&\leq C\big\|f\big\|_{\mathcal M^{p,\theta}(w^p,w^q)}.
\end{split}
\end{equation*}
Let us now turn to the estimate of $J_2$. By definition, for any $x\in B$, we have
\begin{equation*}
\big|[b,I_{\alpha}](f_2)(x)\big|\leq\big|b(x)-b_{B}\big|\cdot\big|I_{\alpha}(f_2)(x)\big|
+\Big|I_{\alpha}\big([b_{B}-b]f_2\big)(x)\Big|.
\end{equation*}
In the proof of Theorem \ref{mainthm:1}, we have already shown that (see \eqref{pointwise1})
\begin{equation*}
\big|I_{\alpha}(f_2)(x)\big|\leq C\sum_{j=1}^\infty\frac{1}{|2^{j+1}B|^{1-{\alpha}/n}}\int_{2^{j+1}B}|f(y)|\,dy.
\end{equation*}
Following the same argument as in \eqref{pointwise1}, we can also prove that
\begin{align}\label{pointwise2}
\Big|I_{\alpha}\big([b_{B}-b]f_2\big)(x)\Big|
&\leq\int_{\mathbb R^n}\frac{|[b_B-b(y)]f_2(y)|}{|x-y|^{n-\alpha}}dy\notag\\
&\leq C\int_{(2B)^c}\frac{|[b_B-b(y)]f(y)|}{|x_0-y|^{n-\alpha}}dy\notag\\
&\leq C\sum_{j=1}^\infty\frac{1}{|2^{j+1}B|^{1-{\alpha}/n}}\int_{2^{j+1}B}\big|b(y)-b_{B}\big|\cdot\big|f(y)\big|\,dy.
\end{align}
Hence, from the above two pointwise estimates for $\big|I_\alpha(f_2)(x)\big|$ and $\big|I_\alpha\big([b_{B}-b]f_2\big)(x)\big|$, it follows that
\begin{equation*}
\begin{split}
J_2&\leq\frac{C}{\theta(w^q(B))^{1/p}}\bigg(\int_B\big|b(x)-b_B\big|^qw^q(x)\,dx\bigg)^{1/q}
\times\bigg(\sum_{j=1}^\infty\frac{1}{|2^{j+1}B|^{1-{\alpha}/n}}\int_{2^{j+1}B}|f(y)|\,dy\bigg)\\
&+C\cdot\frac{w^q(B)^{1/q}}{\theta(w^q(B))^{1/p}}\sum_{j=1}^\infty
\frac{1}{|2^{j+1}B|^{1-{\alpha}/n}}\int_{2^{j+1}B}\big|b_{2^{j+1}B}-b_B\big|\cdot\big|f(y)\big|\,dy\\
&+C\cdot\frac{w^q(B)^{1/q}}{\theta(w^q(B))^{1/p}}\sum_{j=1}^\infty
\frac{1}{|2^{j+1}B|^{1-{\alpha}/n}}\int_{2^{j+1}B}\big|b(y)-b_{2^{j+1}B}\big|\cdot\big|f(y)\big|\,dy\\
&:=J_3+J_4+J_5.
\end{split}
\end{equation*}
Below we will give the estimates of $J_3$, $J_4$ and $J_5$, respectively. To estimate $J_3$, note that $w^q\in A_q\subset A_{\infty}$ with $1<q<\infty$. Using the second part of Lemma \ref{BMO}, H\"older's inequality and the $A_{p,q}$ condition on $w$, we obtain
\begin{equation*}
\begin{split}
J_3&\leq C\|b\|_*\cdot\frac{w^q(B)^{1/q}}{\theta(w^q(B))^{1/p}}
\times\bigg(\sum_{j=1}^\infty\frac{1}{|2^{j+1}B|^{1-{\alpha}/n}}\int_{2^{j+1}B}|f(y)|\,dy\bigg)\\
&\leq C\|b\|_*\cdot\frac{w^q(B)^{1/q}}{\theta(w^q(B))^{1/p}}
\sum_{j=1}^\infty\frac{1}{|2^{j+1}B|^{1-{\alpha}/n}}
\bigg(\int_{2^{j+1}B}\big|f(y)\big|^pw^p(y)\,dy\bigg)^{1/p}\\
&\times\left(\int_{2^{j+1}B}w(y)^{-{p'}}dy\right)^{1/{p'}}\\
&\leq C\big\|f\big\|_{\mathcal M^{p,\theta}(w^p,w^q)}
\times\sum_{j=1}^\infty\frac{\theta(w^q(2^{j+1}B))^{1/p}}{\theta(w^q(B))^{1/p}}\cdot\frac{w^q(B)^{1/q}}{w^q(2^{j+1}B)^{1/q}}\\
&\leq C\big\|f\big\|_{\mathcal M^{p,\theta}(w^p,w^q)},
\end{split}
\end{equation*}
where in the last inequality we have used the estimate \eqref{theta1}. To estimate $J_4$, applying the first part of Lemma \ref{BMO}, H\"older's inequality and the $A_{p,q}$ condition on $w$, we can deduce that
\begin{equation*}
\begin{split}
J_4&\leq C\|b\|_*\cdot\frac{w^q(B)^{1/q}}{\theta(w^q(B))^{1/p}}
\times\sum_{j=1}^\infty\frac{(j+1)}{|2^{j+1}B|^{1-{\alpha}/n}}\int_{2^{j+1}B}|f(y)|\,dy\\
&\leq C\|b\|_*\cdot\frac{w^q(B)^{1/q}}{\theta(w^q(B))^{1/p}}
\sum_{j=1}^\infty\frac{(j+1)}{|2^{j+1}B|^{1-{\alpha}/n}}
\bigg(\int_{2^{j+1}B}\big|f(y)\big|^pw^p(y)\,dy\bigg)^{1/p}\\
&\times\left(\int_{2^{j+1}B}w(y)^{-{p'}}dy\right)^{1/{p'}}\\
&\leq C\big\|f\big\|_{\mathcal M^{p,\theta}(w^p,w^q)}
\times\sum_{j=1}^\infty\big(j+1\big)
\cdot\frac{\theta(w^q(2^{j+1}B))^{1/p}}{\theta(w^q(B))^{1/p}}\cdot\frac{w^q(B)^{1/q}}{w^q(2^{j+1}B)^{1/q}}.
\end{split}
\end{equation*}
For any $j\in\mathbb Z^+$, since $0<w^q(B)<w^q(2^{j+1}B)<+\infty$ when $w^q\in A_q$ with $1<q<\infty$, then by using the $\mathcal D_\kappa$ condition (\ref{D condition}) of $\theta$ and the inequality \eqref{compare} with exponent $\delta>0$, we thus obtain
\begin{align}\label{theta3}
\sum_{j=1}^\infty\big(j+1\big)\cdot\frac{\theta(w^q(2^{j+1}B))^{1/p}}{\theta(w^q(B))^{1/p}}\cdot\frac{w^q(B)^{1/q}}{w^q(2^{j+1}B)^{1/q}}
&\leq C\sum_{j=1}^\infty\big(j+1\big)\cdot\frac{w^q(B)^{1/q-{\kappa}/p}}{w^q(2^{j+1}B)^{1/q-{\kappa}/p}}\notag\\
&\leq C\sum_{j=1}^\infty\big(j+1\big)\cdot\left(\frac{|B|}{|2^{j+1}B|}\right)^{\delta{(1/q-\kappa/p)}}\notag\\
&\leq C\sum_{j=1}^\infty\big(j+1\big)\cdot\left(\frac{1}{2^{(j+1)n}}\right)^{\delta{(1/q-\kappa/p)}}\notag\\
&\leq C,
\end{align}
where the last series is convergent since the exponent $\delta{(1/q-\kappa/p)}$ is positive. This implies our desired estimate $J_4\leq C\big\|f\big\|_{\mathcal M^{p,\theta}(w^p,w^q)}$. It remains to estimate the last term $J_5$. An application of H\"older's inequality gives us that
\begin{equation*}
\begin{split}
J_5&\leq C\cdot\frac{w^q(B)^{1/q}}{\theta(w^q(B))^{1/p}}\sum_{j=1}^\infty\frac{1}{|2^{j+1}B|^{1-{\alpha}/n}}
\bigg(\int_{2^{j+1}B}\big|f(y)\big|^pw^p(y)\,dy\bigg)^{1/p}\\
&\times\left(\int_{2^{j+1}B}\big|b(y)-b_{2^{j+1}B}\big|^{p'}w(y)^{-{p'}}dy\right)^{1/{p'}}.
\end{split}
\end{equation*}
If we set $\mu(y)=w(y)^{-{p'}}$, then we have $\mu\in A_{p'}\subset A_{\infty}$ because $w\in A_{p,q}$ by Lemma \ref{relation}$(i)$. Thus, it follows from the second part of Lemma \ref{BMO} and the $A_{p,q}$ condition that
\begin{align}\label{WBMO}
\left(\int_{2^{j+1}B}\big|b(y)-b_{2^{j+1}B}\big|^{p'}\mu(y)\,dy\right)^{1/{p'}}
&\leq C\|b\|_*\cdot\mu\big(2^{j+1}B\big)^{1/{p'}}\notag\\
&=C\|b\|_*\cdot\left(\int_{2^{j+1}B}w(y)^{-{p'}}dy\right)^{1/{p'}}\notag\\
&\leq C\|b\|_*\cdot\frac{|2^{j+1}B|^{1-{\alpha}/n}}{w^q(2^{j+1}B)^{1/q}}.
\end{align}
Therefore, in view of the estimates \eqref{WBMO} and \eqref{theta1}, we conclude that
\begin{equation*}
\begin{split}
J_5&\leq C\|b\|_*\cdot\frac{w^q(B)^{1/q}}{\theta(w^q(B))^{1/p}}
\sum_{j=1}^\infty\frac{1}{w^q(2^{j+1}B)^{1/q}}\bigg(\int_{2^{j+1}B}\big|f(y)\big|^pw^p(y)\,dy\bigg)^{1/p}\\
&\leq C\big\|f\big\|_{\mathcal M^{p,\theta}(w^p,w^q)}\times
\sum_{j=1}^\infty\frac{\theta(w^q(2^{j+1}B))^{1/p}}{\theta(w^q(B))^{1/p}}\cdot\frac{w^q(B)^{1/q}}{w^q(2^{j+1}B)^{1/q}}\\
&\leq C\big\|f\big\|_{\mathcal M^{p,\theta}(w^p,w^q)}.
\end{split}
\end{equation*}
Summarizing the estimates derived above and then taking the supremum over all balls $B\subset\mathbb R^n$, we complete the proof of Theorem \ref{mainthm:3}.
\end{proof}

\begin{proof}[Proof of Theorem $\ref{mainthm:4}$]
For any fixed ball $B=B(x_0,r_B)$ in $\mathbb R^n$, as before, we represent $f$ as $f=f_1+f_2$, where $f_1=f\cdot\chi_{2B}$, $2B=B(x_0,2r_B)\subset\mathbb R^n$. Then for any given $\sigma>0$, by the linearity of the commutator operator $[b,I_{\alpha}]$, we write
\begin{equation*}
\begin{split}
&\frac{1}{\theta(w^q(B))}\cdot \Big[w^q\big(\big\{x\in B:\big|[b,I_{\alpha}](f)(x)\big|>\sigma\big\}\big)\Big]^{1/q}\\
\leq &\frac{1}{\theta(w^q(B))}\cdot \Big[w^q\big(\big\{x\in B:\big|[b,I_{\alpha}](f_1)(x)\big|>\sigma/2\big\}\big)\Big]^{1/q}\\
&+\frac{1}{\theta(w^q(B))}\cdot \Big[w^q\big(\big\{x\in B:\big|[b,I_{\alpha}](f_2)(x)\big|>\sigma/2\big\}\big)\Big]^{1/q}\\
:=&J'_1+J'_2.
\end{split}
\end{equation*}
We first consider the term $J'_1$. By using Theorem \ref{cweak} and the previous estimate \eqref{main esti2}, we get
\begin{equation*}
\begin{split}
J'_1&\leq C\cdot\frac{1}{\theta(w^q(B))}\int_{\mathbb R^n}\Phi\left(\frac{|f_1(x)|}{\sigma}\right)\cdot w(x)\,dx\\
&= C\cdot\frac{1}{\theta(w^q(B))}\int_{2B}\Phi\left(\frac{|f(x)|}{\sigma}\right)\cdot w(x)\,dx\\
&= C\cdot\frac{\theta(w^q(2B))}{\theta(w^q(B))}\cdot\frac{1}{\theta(w^q(2B))}
\int_{2B}\Phi\left(\frac{|f(x)|}{\sigma}\right)\cdot w(x)\,dx\\
&\leq C\cdot\frac{\theta(w^q(2B))}{\theta(w^q(B))}\cdot\frac{w(2B)}{\theta(w^q(2B))}
\cdot\bigg\|\Phi\left(\frac{|f|}{\,\sigma\,}\right)\bigg\|_{L\log L(w),2B}.
\end{split}
\end{equation*}
Since $w$ is a weight in the class $A_{1,q}$, one has $w^q\in A_1\subset A_{\infty}$ by Lemma \ref{relation}$(ii)$. Moreover, since $0<w^q(B)<w^q(2B)<+\infty$ when $w^q\in A_1$, then by the $\mathcal D_\kappa$ condition (\ref{D condition}) of $\theta$ and inequality (\ref{weights}), we have
\begin{equation*}
\begin{split}
J'_1&\leq C\cdot\frac{w^q(2B)^\kappa}{w^q(B)^\kappa}\cdot
\left\{\frac{w(2B)}{\theta(w^q(2B))}\cdot\bigg\|\Phi\left(\frac{|f|}{\,\sigma\,}\right)\bigg\|_{L\log L(w),2B}\right\}\\
&\leq C\cdot\bigg\|\Phi\left(\frac{|f|}{\,\sigma\,}\right)\bigg\|_{\mathcal M^{1,\theta}_{L\log L}(w,w^q)},
\end{split}
\end{equation*}
which is our desired estimate. We now turn to deal with the term $J'_2$. Recall that the following inequality
\begin{equation*}
\big|[b,I_{\alpha}](f_2)(x)\big|\leq\big|b(x)-b_{B}\big|\cdot\big|I_{\alpha}(f_2)(x)\big|
+\Big|I_{\alpha}\big([b_{B}-b]f_2\big)(x)\Big|
\end{equation*}
is valid. So we can further decompose $J'_2$ as
\begin{equation*}
\begin{split}
J'_2\leq&\frac{1}{\theta(w^q(B))}\cdot
\Big[w^q\big(\big\{x\in B:\big|b(x)-b_{B}\big|\cdot\big|I_{\alpha}(f_2)(x)\big|>\sigma/4\big\}\big)\Big]^{1/q}\\
&+\frac{1}{\theta(w^q(B))}\cdot
\Big[w^q\Big(\Big\{x\in B:\Big|I_{\alpha}\big([b_{B}-b]f_2\big)(x)\Big|>\sigma/4\Big\}\Big)\Big]^{1/q}\\
:=&J'_3+J'_4.
\end{split}
\end{equation*}
By using the previous pointwise estimate \eqref{pointwise1}, Chebyshev's inequality together with Lemma \ref{BMO}$(ii)$, we deduce that
\begin{equation*}
\begin{split}
J'_3&\leq\frac{1}{\theta(w^q(B))}\cdot\frac{\,4\,}{\sigma}
\bigg(\int_B\big|b(x)-b_{B}\big|^q\cdot\big|I_{\alpha}(f_2)(x)\big|^qw^q(x)\,dx\bigg)^{1/q}\\
&\leq C\sum_{j=1}^\infty\frac{1}{|2^{j+1}B|^{1-{\alpha}/n}}\int_{2^{j+1}B}\frac{|f(y)|}{\sigma}\,dy
\times\frac{1}{\theta(w^q(B))}\cdot\bigg(\int_B\big|b(x)-b_{B}\big|^qw^q(x)\,dx\bigg)^{1/q}\\
&\leq C\|b\|_*\sum_{j=1}^\infty\frac{1}{|2^{j+1}B|^{1-{\alpha}/n}}\int_{2^{j+1}B}\frac{|f(y)|}{\sigma}\,dy
\times\frac{w^q(B)^{1/q}}{\theta(w^q(B))}.\\
\end{split}
\end{equation*}
Furthermore, note that $t\leq\Phi(t)=t\cdot(1+\log^+t)$ for any $t>0$. As we pointed out in Theorem \ref{mainthm:2} that $w^q\in A_1$ if and only if $w\in A_1\cap RH_q$, it then follows from the $A_1$ condition and the previous estimate \eqref{main esti1} that
\begin{equation*}
\begin{split}
J'_3&\leq C\sum_{j=1}^\infty\frac{1}{|2^{j+1}B|^{1-{\alpha}/n}}\int_{2^{j+1}B}\Phi\left(\frac{|f(y)|}{\sigma}\right)dy
\times\frac{w^q(B)^{1/q}}{\theta(w^q(B))}\\
&\leq C\sum_{j=1}^\infty\frac{|2^{j+1}B|^{{\alpha}/n}}{w(2^{j+1}B)}\int_{2^{j+1}B}\Phi\left(\frac{|f(y)|}{\sigma}\right)\cdot w(y)\,dy
\times\frac{w^q(B)^{1/q}}{\theta(w^q(B))}\\
&\leq C\sum_{j=1}^\infty\bigg\|\Phi\left(\frac{|f|}{\,\sigma\,}\right)\bigg\|_{L\log L(w),2^{j+1}B}
\times|2^{j+1}B|^{{\alpha}/n}\cdot\frac{w^q(B)^{1/q}}{\theta(w^q(B))}\\
&= C\sum_{j=1}^\infty\left\{\frac{w(2^{j+1}B)}{\theta(w^q(2^{j+1}B))}\cdot
\bigg\|\Phi\left(\frac{|f|}{\,\sigma\,}\right)\bigg\|_{L\log L(w),2^{j+1}B}\right\}\\
&\times\frac{|2^{j+1}B|^{{\alpha}/n}}{w(2^{j+1}B)}\cdot\frac{\theta(w^q(2^{j+1}B))}{\theta(w^q(B))}\cdot w^q(B)^{1/q}.\\
\end{split}
\end{equation*}
In view of \eqref{wq} and \eqref{theta2}, we have
\begin{equation*}
\begin{split}
J'_3&\leq C\cdot\bigg\|\Phi\left(\frac{|f|}{\,\sigma\,}\right)\bigg\|_{\mathcal M^{1,\theta}_{L\log L}(w,w^q)}
\times\sum_{j=1}^\infty\frac{|2^{j+1}B|^{{\alpha}/n}}{w(2^{j+1}B)}\cdot\frac{\theta(w^q(2^{j+1}B))}{\theta(w^q(B))}\cdot w^q(B)^{1/q}\\
&\leq C\cdot\bigg\|\Phi\left(\frac{|f|}{\,\sigma\,}\right)\bigg\|_{\mathcal M^{1,\theta}_{L\log L}(w,w^q)}
\times\sum_{j=1}^\infty\frac{\theta(w^q(2^{j+1}B))}{\theta(w^q(B))}\cdot\frac{w^q(B)^{1/q}}{w^q(2^{j+1}B)^{1/q}}\\
&\leq C\cdot\bigg\|\Phi\left(\frac{|f|}{\,\sigma\,}\right)\bigg\|_{\mathcal M^{1,\theta}_{L\log L}(w,w^q)}.
\end{split}
\end{equation*}
On the other hand, applying the pointwise estimate \eqref{pointwise2} and Chebyshev's inequality, we get
\begin{equation*}
\begin{split}
J'_4&\leq\frac{1}{\theta(w^q(B))}\cdot\frac{\,4\,}{\sigma}
\bigg(\int_B\Big|I_{\alpha}\big([b_{B}-b]f_2\big)(x)\Big|^qw^q(x)\,dx\bigg)^{1/q}\\
&\leq\frac{w^q(B)^{1/q}}{\theta(w^q(B))}\cdot\frac{\,C\,}{\sigma}
\sum_{j=1}^\infty\frac{1}{|2^{j+1}B|^{1-{\alpha}/n}}\int_{2^{j+1}B}\big|b(y)-b_{B}\big|\cdot\big|f(y)\big|\,dy\\
\end{split}
\end{equation*}
\begin{equation*}
\begin{split}
&\leq\frac{w^q(B)^{1/q}}{\theta(w^q(B))}\cdot\frac{\,C\,}{\sigma}
\sum_{j=1}^\infty\frac{1}{|2^{j+1}B|^{1-{\alpha}/n}}\int_{2^{j+1}B}\big|b(y)-b_{2^{j+1}B}\big|\cdot\big|f(y)\big|\,dy\\
&+\frac{w^q(B)^{1/q}}{\theta(w^q(B))}\cdot\frac{\,C\,}{\sigma}
\sum_{j=1}^\infty\frac{1}{|2^{j+1}B|^{1-{\alpha}/n}}\int_{2^{j+1}B}\big|b_{2^{j+1}B}-b_B\big|\cdot\big|f(y)\big|\,dy\\
&:=J'_5+J'_6.
\end{split}
\end{equation*}
For the term $J'_5$, since $w\in A_1$, it follows from the $A_1$ condition and the fact $t\leq \Phi(t)$ that
\begin{equation*}
\begin{split}
J'_5&\leq\frac{\,C\,}{\sigma}\cdot\frac{w^q(B)^{1/q}}{\theta(w^q(B))}
\sum_{j=1}^\infty\frac{|2^{j+1}B|^{{\alpha}/n}}{w(2^{j+1}B)}\int_{2^{j+1}B}\big|b(y)-b_{2^{j+1}B}\big|\cdot\big|f(y)\big|w(y)\,dy\\
&\leq C\cdot\frac{w^q(B)^{1/q}}{\theta(w^q(B))}
\sum_{j=1}^\infty\frac{|2^{j+1}B|^{{\alpha}/n}}{w(2^{j+1}B)}\int_{2^{j+1}B}\big|b(y)-b_{2^{j+1}B}\big|
\cdot\Phi\left(\frac{|f(y)|}{\sigma}\right)w(y)\,dy.\\
\end{split}
\end{equation*}
Furthermore, we use the generalized H\"older's inequality with weight \eqref{Wholder} to obtain
\begin{equation*}
\begin{split}
J'_5&\leq C\cdot\frac{w^q(B)^{1/q}}{\theta(w^q(B))}
\sum_{j=1}^\infty\big|2^{j+1}B\big|^{{\alpha}/n}\cdot\big\|b-b_{2^{j+1}B}\big\|_{\exp L(w),2^{j+1}B}
\bigg\|\Phi\left(\frac{|f|}{\,\sigma\,}\right)\bigg\|_{L\log L(w),2^{j+1}B}\\
&\leq C\|b\|_*\cdot\frac{w^q(B)^{1/q}}{\theta(w^q(B))}
\sum_{j=1}^\infty\big|2^{j+1}B\big|^{{\alpha}/n}\cdot\bigg\|\Phi\left(\frac{|f|}{\,\sigma\,}\right)\bigg\|_{L\log L(w),2^{j+1}B}.
\end{split}
\end{equation*}
In the last inequality, we have used the well-known fact that (see \cite{zhang})
\begin{equation}\label{Jensen}
\big\|b-b_{B}\big\|_{\exp L(w),B}\leq C\|b\|_*,\qquad \mbox{for any ball }B\subset\mathbb R^n.
\end{equation}
It is equivalent to the inequality
\begin{equation*}
\frac{1}{w(B)}\int_B\exp\bigg(\frac{|b(y)-b_B|}{c_0\|b\|_*}\bigg)w(y)\,dy\leq C,
\end{equation*}
which is just a corollary of the well-known John--Nirenberg's inequality (see \cite{john}) and the comparison property of $A_1$ weights. Hence, by the estimates \eqref{wq} and \eqref{theta2},
\begin{equation*}
\begin{split}
J'_5&\leq C\|b\|_*\sum_{j=1}^\infty\left\{\frac{w(2^{j+1}B)}{\theta(w^q(2^{j+1}B))}\cdot
\bigg\|\Phi\left(\frac{|f|}{\,\sigma\,}\right)\bigg\|_{L\log L(w),2^{j+1}B}\right\}\\
&\times\frac{|2^{j+1}B|^{{\alpha}/n}}{w(2^{j+1}B)}\cdot\frac{\theta(w^q(2^{j+1}B))}{\theta(w^q(B))}\cdot w^q(B)^{1/q}\\
&\leq C\cdot\bigg\|\Phi\left(\frac{|f|}{\,\sigma\,}\right)\bigg\|_{\mathcal M^{1,\theta}_{L\log L}(w,w^q)}
\times\sum_{j=1}^\infty\frac{\theta(w^q(2^{j+1}B))}{\theta(w^q(B))}\cdot\frac{w^q(B)^{1/q}}{w^q(2^{j+1}B)^{1/q}}\\
&\leq C\cdot\bigg\|\Phi\left(\frac{|f|}{\,\sigma\,}\right)\bigg\|_{\mathcal M^{1,\theta}_{L\log L}(w,w^q)}.
\end{split}
\end{equation*}
For the last term $J'_6$ we proceed as follows. Using the first part of Lemma \ref{BMO} together with the facts $w\in A_1$ and $t\leq\Phi(t)=t\cdot(1+\log^+t)$, we deduce that
\begin{equation*}
\begin{split}
J'_6&\leq C\cdot\frac{w^q(B)^{1/q}}{\theta(w^q(B))}
\sum_{j=1}^\infty(j+1)\|b\|_*\cdot\frac{1}{|2^{j+1}B|^{1-{\alpha}/n}}\int_{2^{j+1}B}\frac{|f(y)|}{\sigma}\,dy\\
&\leq C\cdot\frac{w^q(B)^{1/q}}{\theta(w^q(B))}
\sum_{j=1}^\infty(j+1)\|b\|_*\cdot\frac{|2^{j+1}B|^{{\alpha}/n}}{w(2^{j+1}B)}\int_{2^{j+1}B}\frac{|f(y)|}{\sigma}\cdot w(y)\,dy\\
&\leq C\|b\|_*\cdot\frac{w^q(B)^{1/q}}{\theta(w^q(B))}
\sum_{j=1}^\infty\frac{(j+1)|2^{j+1}B|^{{\alpha}/n}}{w(2^{j+1}B)}\int_{2^{j+1}B}\Phi\left(\frac{|f(y)|}{\sigma}\right)\cdot w(y)\,dy.\\
\end{split}
\end{equation*}
Making use of the inequalities \eqref{main esti1} and \eqref{wq}, we further obtain
\begin{equation*}
\begin{split}
J'_6&\leq C\cdot\frac{w^q(B)^{1/q}}{\theta(w^q(B))}
\sum_{j=1}^\infty(j+1)|2^{j+1}B|^{{\alpha}/n}\cdot\bigg\|\Phi\left(\frac{|f|}{\,\sigma\,}\right)\bigg\|_{L\log L(w),2^{j+1}B}\\
&=C\cdot\sum_{j=1}^\infty\left\{\frac{w(2^{j+1}B)}{\theta(w^q(2^{j+1}B))}\cdot
\bigg\|\Phi\left(\frac{|f|}{\,\sigma\,}\right)\bigg\|_{L\log L(w),2^{j+1}B}\right\}\\
&\times(j+1)\cdot\frac{|2^{j+1}B|^{{\alpha}/n}}{w(2^{j+1}B)}\cdot\frac{\theta(w^q(2^{j+1}B))}{\theta(w^q(B))}\cdot w^q(B)^{1/q}\\
&\leq C\cdot\bigg\|\Phi\left(\frac{|f|}{\,\sigma\,}\right)\bigg\|_{\mathcal M^{1,\theta}_{L\log L}(w,w^q)}
\times\sum_{j=1}^\infty(j+1)\cdot\frac{\theta(w^q(2^{j+1}B))}{\theta(w^q(B))}\cdot\frac{w^q(B)^{1/q}}{w^q(2^{j+1}B)^{1/q}}.
\end{split}
\end{equation*}
Recall that $w^q\in A_1\subset A_\infty$ with $1<q<\infty$. We can now argue exactly as we did in the estimation of $J_4$ to get (now choose $\delta^*$ in \eqref{compare})
\begin{align}\label{theta4}
\sum_{j=1}^\infty\big(j+1\big)\cdot\frac{\theta(w^q(2^{j+1}B))}{\theta(w^q(B))}\cdot\frac{w^q(B)^{1/q}}{w^q(2^{j+1}B)^{1/q}}
&\leq C\sum_{j=1}^\infty\big(j+1\big)\cdot\frac{w^q(B)^{1/q-\kappa}}{w^q(2^{j+1}B)^{1/q-\kappa}}\notag\\
&\leq C\sum_{j=1}^\infty\big(j+1\big)\cdot\left(\frac{|B|}{|2^{j+1}B|}\right)^{\delta^*{(1/q-\kappa)}}\notag\\
&\leq C\sum_{j=1}^\infty\big(j+1\big)\cdot\left(\frac{1}{2^{(j+1)n}}\right)^{\delta^*{(1/q-\kappa)}}\notag\\
&\leq C.
\end{align}
Notice that the exponent $\delta^*{(1/q-\kappa)}$ is positive by the choice of $\kappa$, which guarantees that the last series is convergent. If we substitute this estimate \eqref{theta4} into the term $J'_6$, then we get the desired inequality
\begin{equation*}
J'_6\leq C\cdot\bigg\|\Phi\left(\frac{|f|}{\,\sigma\,}\right)\bigg\|_{\mathcal M^{1,\theta}_{L\log L}(w,w^q)}.
\end{equation*}
This completes the proof of Theorem \ref{mainthm:4}.
\end{proof}

\section{Proofs of Theorems \ref{mainthm:end1} and \ref{mainthm:end2}}

\begin{proof}[Proof of Theorem $\ref{mainthm:end1}$]
Let $f\in\mathcal M^{p,\theta}(w^p,w^q)$ with $1<p,q<\infty$ and $w\in A_{p,q}$. For any given ball $B=B(x_0,r_B)$ in $\mathbb R^n$, it suffices to prove that the following inequality
\begin{equation}\label{end1.1}
\frac{1}{|B|}\int_B|I_{\alpha}f(x)-(I_{\alpha}f)_B|\,dx\leq C\big\|f\big\|_{\mathcal L^{p,\kappa}(w^p,w^q)}
\end{equation}
holds. Decompose $f$ as $f=f_1+f_2$, where $f_1=f\cdot\chi_{4B}$, $f_2=f\cdot\chi_{(4B)^c}$, $4B=B(x_0,4r_B)$. By the linearity of the fractional integral operator $I_{\alpha}$, the left-hand side of \eqref{end1.1} can be divided into two parts. That is,
\begin{equation*}
\begin{split}
&\frac{1}{|B|}\int_B|I_{\alpha}f(x)-(I_{\alpha}f)_B|\,dx\\
&\leq \frac{1}{|B|}\int_B|I_{\alpha}f_1(x)-(I_{\alpha}f_1)_B|\,dx+\frac{1}{|B|}\int_B|I_{\alpha}f_2(x)-(I_{\alpha}f_2)_B|\,dx\\
&:=I+II.
\end{split}
\end{equation*}
First let us consider the term $I$. Applying the weighted $(L^p,L^q)$-boundedness of $I_{\alpha}$ (see Theorem \ref{strong}) and H\"older's inequality, we obtain
\begin{equation*}
\begin{split}
I&\leq\frac{2}{|B|}\int_B|I_{\alpha}f_1(x)|\,dx\\
&\leq\frac{2}{|B|}\bigg(\int_B|I_{\alpha}f_1(x)|^qw^q(x)\,dx\bigg)^{1/q}\bigg(\int_B w(x)^{-q'}dx\bigg)^{1/{q'}}\\
&\leq\frac{C}{|B|}\bigg(\int_{4B}|f(x)|^pw^p(x)\,dx\bigg)^{1/p}\bigg(\int_B w(x)^{-q'}dx\bigg)^{1/{q'}}\\
&\leq C\big\|f\big\|_{\mathcal L^{p,\kappa}(w^p,w^q)}
\cdot\frac{w^q(4B)^{{\kappa}/p}}{|B|}\bigg(\int_B w(x)^{-q'}dx\bigg)^{1/{q'}}.
\end{split}
\end{equation*}
Since $w$ is a weight in the class $A_{p,q}$, one has $w^q\in A_q\subset A_{\infty}$ by Lemma \ref{relation}$(i)$. By definition, it reads
\begin{equation*}
\left(\frac1{|B|}\int_B w^q(x)\,dx\right)^{1/q}\left(\frac1{|B|}\int_B [w^q(x)]^{-q'/q}\,dx\right)^{1/{q'}}\leq C,
\end{equation*}
which implies
\begin{equation}\label{end1.2}
\bigg(\int_B w(x)^{-q'}dx\bigg)^{1/{q'}}\leq C\cdot\frac{|B|}{w^q(B)^{1/q}}.
\end{equation}
Since $w^q\in A_q\subset A_{\infty}$, then $w^q\in\Delta_2$. Using the inequalities \eqref{end1.2} and \eqref{weights} and noting the fact that $\kappa=p/q$, we have
\begin{equation*}
\begin{split}
I&\leq C\big\|f\big\|_{\mathcal L^{p,\kappa}(w^p,w^q)}\cdot\frac{w^q(4B)^{1/q}}{w^q(B)^{1/q}}\\
&\leq C\big\|f\big\|_{\mathcal L^{p,\kappa}(w^p,w^q)}.
\end{split}
\end{equation*}
Now we estimate $II$. For any $x\in B$,
\begin{equation*}
\begin{split}
|I_{\alpha}f_2(x)-(I_{\alpha}f_2)_B|&=\bigg|\frac{1}{|B|}\int_B\big[I_{\alpha}f_2(x)-I_{\alpha}f_2(y)\big]\,dy\bigg|\\
&=\bigg|\frac{1}{|B|}\int_B\bigg\{\int_{(4B)^c}\bigg[\frac{1}{|x-z|^{n-\alpha}}-\frac{1}{|y-z|^{n-\alpha}}\bigg]f(z)\,dz\bigg\}dy\bigg|\\
&\leq\frac{1}{|B|}\int_B\bigg\{\int_{(4B)^c}\bigg|\frac{1}{|x-z|^{n-\alpha}}-\frac{1}{|y-z|^{n-\alpha}}\bigg|\cdot|f(z)|\,dz\bigg\}dy.
\end{split}
\end{equation*}
Since both $x$ and $y$ are in $B$, $z\in(4B)^c$, by a purely geometric observation, we must have $|x-z|\geq 2|x-y|$. This fact along with the mean value theorem yields
\begin{align}\label{average}
|I_{\alpha}f_2(x)-(I_{\alpha}f_2)_B|
&\leq\frac{C}{|B|}\int_B\bigg\{\int_{(4B)^c}\frac{|x-y|}{|x-z|^{n-\alpha+1}}\cdot|f(z)|\,dz\bigg\}dy\notag\\
&\leq C\int_{(4B)^c}\frac{r_B}{|z-x_0|^{n-\alpha+1}}\cdot|f(z)|\,dz\notag\\
&\leq C\sum_{j=2}^\infty\frac{1}{2^j}\cdot\frac{1}{|2^{j+1}B|^{1-{\alpha}/n}}\int_{2^{j+1}B}|f(z)|\,dz.
\end{align}
Furthermore, by using H\"older's inequality and $A_{p,q}$ condition on $w$, we get for any $x\in B$,
\begin{align}\label{end1.3}
|I_{\alpha}f_2(x)-(I_{\alpha}f_2)_B|&\leq C\sum_{j=2}^\infty\frac{1}{2^j}\cdot\frac{1}{|2^{j+1}B|^{1-{\alpha}/n}}\notag\\
&\times\bigg(\int_{2^{j+1}B}\big|f(y)\big|^pw^p(y)\,dy\bigg)^{1/p}
\left(\int_{2^{j+1}B}w(y)^{-p'}dy\right)^{1/{p'}}\notag\\
&\leq C\big\|f\big\|_{\mathcal L^{p,\kappa}(w^p,w^q)}\cdot
\sum_{j=2}^\infty\frac{1}{2^j}\cdot\frac{w^q(2^{j+1}B)^{{\kappa}/p}}{w^q(2^{j+1}B)^{1/q}}\notag\\
&=C\big\|f\big\|_{\mathcal L^{p,\kappa}(w^p,w^q)}\cdot
\sum_{j=2}^\infty\frac{1}{2^j}\notag\\
&\leq C\big\|f\big\|_{\mathcal L^{p,\kappa}(w^p,w^q)}.
\end{align}
From the pointwise estimate \eqref{end1.3}, it readily follows that
\begin{equation*}
II=\frac{1}{|B|}\int_B|I_{\alpha}f_2(x)-(I_{\alpha}f_2)_B|\,dx\leq C\big\|f\big\|_{\mathcal L^{p,\kappa}(w^p,w^q)}.
\end{equation*}
By combining the above estimates for $I$ and $II$, we are done.
\end{proof}

\begin{proof}[Proof of Theorem $\ref{mainthm:end2}$]
Let $f\in\mathcal L^{p,\Theta}(\mathbb R^n)$ with $1<p<\infty$. For any given ball $B=B(x_0,r_B)$ in $\mathbb R^n$, it is sufficient to prove that the following inequality
\begin{equation}\label{end2.1}
\frac{1}{|B(x_0,r_B)|}\int_{B(x_0,r_B)}|I_{\alpha}f(x)-(I_{\alpha}f)_B|\,dx\leq C\big\|f\big\|_{\mathcal L^{p,\Theta}(\mathbb R^n)}
\end{equation}
holds. Decompose $f$ as $f=f_1+f_2$, where $f_1=f\cdot\chi_{4B}$, $f_2=f\cdot\chi_{(4B)^c}$, $4B=B(x_0,4r_B)$. As in the proof of Theorem \ref{mainthm:end1}, we can also divide the left-hand side of \eqref{end2.1} into two parts. That is,
\begin{equation*}
\begin{split}
&\frac{1}{|B(x_0,r_B)|}\int_{B(x_0,r_B)}|I_{\alpha}f(x)-(I_{\alpha}f)_B|\,dx\\
&\leq\frac{1}{|B(x_0,r_B)|}\int_{B(x_0,r_B)}|I_{\alpha}f_1(x)-(I_{\alpha}f_1)_B|\,dx
+\frac{1}{|B(x_0,r_B)|}\int_{B(x_0,r_B)}|I_{\alpha}f_2(x)-(I_{\alpha}f_2)_B|\,dx\\
&:=I'+II'.
\end{split}
\end{equation*}
First let us consider the term $I'$. Since $I_{\alpha}$ is bounded from $L^p(\mathbb R^n)$ to $L^q(\mathbb R^n)$, then by H\"older's inequality, we obtain
\begin{equation*}
\begin{split}
I'&\leq\frac{2}{|B(x_0,r_B)|}\int_{B(x_0,r_B)}|I_{\alpha}f_1(x)|\,dx\\
&\leq\frac{2}{|B(x_0,r_B)|}\bigg(\int_{B(x_0,r_B)}|I_{\alpha}f_1(x)|^q\,dx\bigg)^{1/q}\bigg(\int_{B(x_0,r_B)}1^{q'}dx\bigg)^{1/{q'}}\\
&\leq\frac{C}{|B(x_0,r_B)|}\bigg(\int_{B(x_0,4r_B)}|f(x)|^p\,dx\bigg)^{1/p}|B(x_0,r_B)|^{1/{q'}}\\
&\leq C\big\|f\big\|_{\mathcal L^{p,\Theta}(\mathbb R^n)}
\cdot\frac{\Theta(4r_B)^{1/p}}{|B(x_0,r_B)|^{1/q}}.
\end{split}
\end{equation*}
Applying our assumption \eqref{assump} on $\Theta$, we further have
\begin{equation*}
I'\leq C\big\|f\big\|_{\mathcal L^{p,\Theta}(\mathbb R^n)}
\cdot\frac{(4r_B)^{n/q}}{|B(x_0,r_B)|^{1/q}}\leq C\big\|f\big\|_{\mathcal L^{p,\Theta}(\mathbb R^n)}.
\end{equation*}
On the other hand, in Theorem \ref{mainthm:end1}, we have already shown that for any $x\in B$ (see \eqref{average})
\begin{equation*}
|I_{\alpha}f_2(x)-(I_{\alpha}f_2)_B|\leq C\sum_{j=2}^\infty\frac{1}{2^j}
\cdot\frac{1}{|B(x_0,2^{j+1}r_B)|^{1-{\alpha}/n}}\int_{B(x_0,2^{j+1}r_B)}|f(z)|\,dz.
\end{equation*}
Moreover, by using H\"older's inequality and the assumption \eqref{assump} on $\Theta$, we can deduce that
\begin{equation*}
\begin{split}
&|I_{\alpha}f_2(x)-(I_{\alpha}f_2)_B|\\
\leq&C\sum_{j=2}^\infty\frac{1}{2^j}\cdot\frac{1}{|B(x_0,2^{j+1}r_B)|^{1-{\alpha}/n}}
\bigg(\int_{B(x_0,2^{j+1}r_B)}|f(z)|^p\,dz\bigg)^{1/p}|B(x_0,2^{j+1}r_B)|^{1/{p'}}\\
\leq& C\big\|f\big\|_{\mathcal L^{p,\Theta}(\mathbb R^n)}\times\sum_{j=2}^\infty\frac{1}{2^j}
\cdot\frac{\Theta(2^{j+1}r_B)^{1/p}}{|B(x_0,2^{j+1}r_B)|^{1/p-{\alpha}/n}}\\
\leq& C\big\|f\big\|_{\mathcal L^{p,\Theta}(\mathbb R^n)}\times\sum_{j=2}^\infty\frac{1}{2^j}
\cdot\frac{(2^{j+1}r_B)^{n/q}}{|B(x_0,2^{j+1}r_B)|^{1/q}}\\
\leq& C\big\|f\big\|_{\mathcal L^{p,\Theta}(\mathbb R^n)}.
\end{split}
\end{equation*}
Therefore,
\begin{equation*}
\begin{split}
II'=\frac{1}{|B(x_0,r_B)|}\int_{B(x_0,r_B)}|I_{\alpha}f_2(x)-(I_{\alpha}f_2)_B|\,dx\leq C\big\|f\big\|_{\mathcal L^{p,\Theta}(\mathbb R^n)}.
\end{split}
\end{equation*}
By combining the above estimates for $I'$ and $II'$, we are done.
\end{proof}

\section{Partial results on two-weight problems}

In the last section, we consider related problems about two-weight, weak type norm inequalities for $I_{\alpha}$ and $[b,I_{\alpha}]$. In \cite{cruz2}, Cruz-Uribe and P\'erez considered the problem of finding sufficient conditions on a pair of weights $(u,v)$ which ensure the boundedness of the operator $I_{\alpha}$ from $L^p(v)$ to $WL^p(u)$, where $1<p<\infty$. They gave a sufficient $A_p$-type condition (see \eqref{assump1.1} below), and proved a two-weight, weak-type $(p,p)$ inequality for $I_{\alpha}$(see also \cite{cruz3} for another, more simpler proof), which solved a problem posed by Sawyer and Wheeden in \cite{sawyer}.
\begin{theorem}[\cite{cruz2,cruz3}]\label{Two1}
Let $0<\alpha<n$ and $1<p<\infty$. Given a pair of weights $(u,v)$, suppose that for some $r>1$ and for all cubes $Q$,
\begin{equation}\label{assump1.1}
\big|Q\big|^{\alpha/n}\cdot\left(\frac{1}{|Q|}\int_Q u(x)^r\,dx\right)^{1/{(rp)}}\left(\frac{1}{|Q|}\int_Q v(x)^{-p'/p}\,dx\right)^{1/{p'}}\leq C<\infty.
\end{equation}
Then the fractional integral operator $I_\alpha$ satisfies the weak-type $(p,p)$ inequality
\begin{equation}\label{assump1.2}
u\big(\big\{x\in\mathbb R^n:\big|I_\alpha f(x)\big|>\sigma\big\}\big)
\leq \frac{C}{\sigma^p}\int_{\mathbb R^n}|f(x)|^p v(x)\,dx,\quad\mbox{for any }~\sigma>0,
\end{equation}
where $C$ does not depend on $f$ and $\sigma>0$.
\end{theorem}

Moreover, in \cite{li}, Li improved this result by replacing the ``power bump" in \eqref{assump1.1} by a smaller ``Orlicz bump". On the other hand, in \cite{liu}, Liu and Lu obtained a sufficient $A_p$-type condition for the commutator $[b,I_{\alpha}]$ to satisfy the two-weight weak type $(p,p)$ inequality, where $1<p<\infty$. That condition is an $A_p$-type condition in the scale of Orlicz spaces (see \eqref{assump2.1} below).
\begin{theorem}[\cite{liu}]\label{Two2}
Let $0<\alpha<n$, $1<p<\infty$ and $b\in BMO(\mathbb R^n)$. Given a pair of weights $(u,v)$, suppose that for some $r>1$ and for all cubes $Q$,
\begin{equation}\label{assump2.1}
\big|Q\big|^{\alpha/n}\cdot\left(\frac{1}{|Q|}\int_Q u(x)^r\,dx\right)^{1/{(rp)}}\big\|v^{-1/p}\big\|_{\mathcal A,Q}\leq C<\infty,
\end{equation}
where $\mathcal A(t)=t^{p'}(1+\log^+t)^{p'}$. Then the linear commutator $[b,I_\alpha]$ satisfies the weak-type $(p,p)$ inequality
\begin{equation}\label{assump2.2}
u\big(\big\{x\in\mathbb R^n:\big|[b,I_\alpha](f)(x)\big|>\sigma\big\}\big)
\leq \frac{C}{\sigma^p}\int_{\mathbb R^n}|f(x)|^p v(x)\,dx,\quad\mbox{for any }~\sigma>0,
\end{equation}
where $C$ does not depend on $f$ and $\sigma>0$.
\end{theorem}
Here and in what follows, all cubes are assumed to have their sides parallel to the coordinate axes, $Q(x_0,\ell)$ will denote the cube centered at $x_0$ and has side length $\ell$. For any cube $Q(x_0,\ell)$ and any $\lambda>0$, we denote by $\lambda Q$ the cube with the same center as $Q$ whose side length is $\lambda$ times that of $Q$, i.e., $\lambda Q:=Q(x_0,\lambda\ell)$. We now extend the results mentioned above to the Morrey type spaces associated to $\theta$.

\begin{theorem}\label{mainthm:5}
Let $0<\alpha<n$ and $1<p<\infty$. Given a pair of weights $(u,v)$, suppose that for some $r>1$ and for all cubes $Q$, \eqref{assump1.1} holds.
If $\theta$ satisfies the $\mathcal D_\kappa$ condition $(\ref{D condition})$ with $0\leq\kappa<1$ and $u\in \Delta_2$, then the fractional integral operator $I_{\alpha}$ is bounded from $\mathcal M^{p,\theta}(v,u)$ into $W\mathcal M^{p,\theta}(u)$.
\end{theorem}

\begin{theorem}\label{mainthm:6}
Let $0<\alpha<n$, $1<p<\infty$ and $b\in BMO(\mathbb R^n)$. Given a pair of weights $(u,v)$, suppose that for some $r>1$ and for all cubes $Q$, \eqref{assump2.1} holds. If $\theta$ satisfies the $\mathcal D_\kappa$ condition $(\ref{D condition})$ with $0\leq\kappa<1$ and $u\in A_\infty$, then the linear commutator $[b,I_{\alpha}]$ is bounded from $\mathcal M^{p,\theta}(v,u)$ into $W\mathcal M^{p,\theta}(u)$.
\end{theorem}

\begin{proof}[Proof of Theorem $\ref{mainthm:5}$]
Let $f\in\mathcal M^{p,\theta}(v,u)$ with $1<p<\infty$. For arbitrary $x_0\in\mathbb R^n$, set $Q=Q(x_0,\ell)$ for the cube centered at $x_0$ and of the side length $\ell$. Let
\begin{equation*}
f=f\cdot\chi_{2Q}+f\cdot\chi_{(2Q)^c}:=f_1+f_2,
\end{equation*}
where $\chi_{2Q}$ denotes the characteristic function of $2Q=Q(x_0,2\ell)$. Then for any given $\sigma>0$, we write
\begin{equation*}
\begin{split}
&\frac{1}{\theta(u(Q))^{1/p}}\sigma\cdot
\Big[u\big(\big\{x\in Q:\big|I_{\alpha}(f)(x)\big|>\sigma\big\}\big)\Big]^{1/p}\\
\leq &\frac{1}{\theta(u(Q))^{1/p}}\sigma\cdot
\Big[u\big(\big\{x\in Q:\big|I_{\alpha}(f_1)(x)\big|>\sigma/2\big\}\big)\Big]^{1/p}\\
&+\frac{1}{\theta(u(Q))^{1/p}}\sigma\cdot
\Big[u\big(\big\{x\in Q:\big|I_{\alpha}(f_2)(x)\big|>\sigma/2\big\}\big)\Big]^{1/p}\\
:=&K_1+K_2.
\end{split}
\end{equation*}
Using Theorem \ref{Two1}, the $\mathcal D_\kappa$ condition (\ref{D condition}) of $\theta$ and inequality (\ref{weights})(consider cube $Q$ instead of ball $B$), we get
\begin{equation*}
\begin{split}
K_1&\leq C\cdot\frac{1}{\theta(u(Q))^{1/p}}\left(\int_{\mathbb R^n}|f_1(x)|^p v(x)\,dx\right)^{1/p}\\
&=C\cdot\frac{1}{\theta(u(Q))^{1/p}}\left(\int_{2Q}|f(x)|^p v(x)\,dx\right)^{1/p}\\
&\leq C\big\|f\big\|_{\mathcal M^{p,\theta}(v,u)}\cdot\frac{\theta(u(2Q))^{1/p}}{\theta(u(Q))^{1/p}}\\
&\leq C\big\|f\big\|_{\mathcal M^{p,\theta}(v,u)}\cdot\frac{u(2Q)^{\kappa/p}}{u(Q)^{\kappa/p}}\\
&\leq C\big\|f\big\|_{\mathcal M^{p,\theta}(v,u)}.
\end{split}
\end{equation*}
As for the term $K_2$, using the same methods and steps as we deal with $I_2$ in Theorem \ref{mainthm:1}, we can also obtain that for any $x\in Q$,
\begin{equation}\label{alpha1}
\big|I_{\alpha}(f_2)(x)\big|\leq C\sum_{j=1}^\infty\frac{1}{|2^{j+1}Q|^{1-{\alpha}/n}}\int_{2^{j+1}Q}|f(y)|\,dy.
\end{equation}
This pointwise estimate \eqref{alpha1} together with Chebyshev's inequality implies
\begin{equation*}
\begin{split}
K_2&\leq\frac{2}{\theta(u(Q))^{1/p}}\cdot\left(\int_Q\big|I_{\alpha}(f_2)(x)\big|^pu(x)\,dx\right)^{1/p}\\
&\leq C\cdot\frac{u(Q)^{1/p}}{\theta(u(Q))^{1/p}}
\sum_{j=1}^\infty\frac{1}{|2^{j+1}Q|^{1-{\alpha}/n}}\int_{2^{j+1}Q}|f(y)|\,dy.
\end{split}
\end{equation*}
Moreover, an application of H\"older's inequality gives that
\begin{equation*}
\begin{split}
K_2&\leq C\cdot\frac{u(Q)^{1/p}}{\theta(u(Q))^{1/p}}
\sum_{j=1}^\infty\frac{1}{|2^{j+1}Q|^{1-{\alpha}/n}}\left(\int_{2^{j+1}Q}|f(y)|^pv(y)\,dy\right)^{1/p}\\
&\times\left(\int_{2^{j+1}Q}v(y)^{-p'/p}dy\right)^{1/{p'}}\\
&\leq C\big\|f\big\|_{\mathcal M^{p,\theta}(v,u)}\cdot\frac{u(Q)^{1/p}}{\theta(u(Q))^{1/p}} \sum_{j=1}^\infty\frac{\theta(u(2^{j+1}Q))^{1/p}}{|2^{j+1}Q|^{1-{\alpha}/n}}
\times\left(\int_{2^{j+1}Q}v(y)^{-p'/p}dy\right)^{1/{p'}}.
\end{split}
\end{equation*}
For any $j\in\mathbb Z^+$, since $0<u(Q)<u(2^{j+1}Q)<+\infty$ when $u$ is a weight function, then by the $\mathcal D_\kappa$ condition (\ref{D condition}) of $\theta$ with $0\leq\kappa<1$, we can see that
\begin{equation}\label{comparethe}
\frac{\theta(u(2^{j+1}Q))^{1/p}}{\theta(u(Q))^{1/p}}\leq\frac{u(2^{j+1}Q)^{\kappa/p}}{u(Q)^{\kappa/p}}.
\end{equation}
In addition, we apply H\"older's inequality with exponent $r$ to get
\begin{equation}\label{U}
u\big(2^{j+1}Q\big)=\int_{2^{j+1}Q}u(y)\,dy
\leq\big|2^{j+1}Q\big|^{1/{r'}}\left(\int_{2^{j+1}Q}u(y)^r\,dy\right)^{1/r}.
\end{equation}
Hence, in view of \eqref{comparethe} and \eqref{U} derived above, we have
\begin{equation*}
\begin{split}
K_2&\leq C\big\|f\big\|_{\mathcal M^{p,\theta}(v,u)}\sum_{j=1}^\infty
\frac{u(Q)^{{(1-\kappa)}/p}}{u(2^{j+1}Q)^{{(1-\kappa)}/p}}\cdot\frac{u(2^{j+1}Q)^{1/p}}{|2^{j+1}Q|^{1-{\alpha}/n}}
\times\left(\int_{2^{j+1}Q}v(y)^{-p'/p}dy\right)^{1/{p'}}\\
&\leq C\big\|f\big\|_{\mathcal M^{p,\theta}(v,u)}\sum_{j=1}^\infty
\frac{u(Q)^{{(1-\kappa)}/p}}{u(2^{j+1}Q)^{{(1-\kappa)}/p}}\cdot\frac{|2^{j+1}Q|^{1/{(r'p)}}}{|2^{j+1}Q|^{1-{\alpha}/n}}\\
&\times\left(\int_{2^{j+1}Q}u(y)^r\,dy\right)^{1/{(rp)}}\left(\int_{2^{j+1}Q}v(y)^{-p'/p}\,dy\right)^{1/{p'}}\\
&\leq C\big\|f\big\|_{\mathcal M^{p,\theta}(v,u)}\times\sum_{j=1}^\infty\frac{u(Q)^{{(1-\kappa)}/p}}{u(2^{j+1}Q)^{{(1-\kappa)}/p}}.
\end{split}
\end{equation*}
The last inequality is obtained by the $A_p$-type condition \eqref{assump1.1} on $(u,v)$. Furthermore, since $u\in\Delta_2$, we can easily check that there exists a reverse doubling constant $D=D(u)>1$ independent of $Q$ such that (see Lemma 4.1 in \cite{komori})
\begin{equation*}
u(2Q)\geq D\cdot u(Q), \quad \mbox{for any cube }\,Q\subset\mathbb R^n,
\end{equation*}
which implies that for any $j\in\mathbb Z^+$, $u(2^{j+1}Q)\geq D^{j+1}\cdot u(Q)$ by iteration. Hence,
\begin{align}\label{C1}
\sum_{j=1}^\infty\frac{u(Q)^{{(1-\kappa)}/p}}{u(2^{j+1}Q)^{{(1-\kappa)}/p}}
&\leq \sum_{j=1}^\infty\left(\frac{u(Q)}{D^{j+1}\cdot u(Q)}\right)^{{(1-\kappa)}/p}\notag\\
&=\sum_{j=1}^\infty\left(\frac{1}{D^{j+1}}\right)^{{(1-\kappa)}/p}\leq C,
\end{align}
where the last series is convergent since the reverse doubling constant $D>1$ and $0\leq\kappa<1$. This yields our desired estimate $K_2\leq C\big\|f\big\|_{\mathcal M^{p,\theta}(v,u)}$. Summing up the above estimates for $K_1$ and $K_2$, and then taking the supremum over all cubes $Q\subset\mathbb R^n$ and all $\sigma>0$, we finish the proof of Theorem \ref{mainthm:5}.
\end{proof}

\begin{proof}[Proof of Theorem $\ref{mainthm:6}$]
Let $f\in\mathcal M^{p,\theta}(v,u)$ with $1<p<\infty$. For an arbitrary cube $Q=Q(x_0,\ell)$ in $\mathbb R^n$, as before, we set
\begin{equation*}
f=f_1+f_2,\qquad f_1=f\cdot\chi_{2Q},\quad  f_2=f\cdot\chi_{(2Q)^c}.
\end{equation*}
Then for any given $\sigma>0$, we write
\begin{equation*}
\begin{split}
&\frac{1}{\theta(u(Q))^{1/p}}\sigma\cdot
\Big[u\big(\big\{x\in Q:\big|[b,I_{\alpha}](f)(x)\big|>\sigma\big\}\big)\Big]^{1/p}\\
\leq &\frac{1}{\theta(u(Q))^{1/p}}\sigma\cdot
\Big[u\big(\big\{x\in Q:\big|[b,I_{\alpha}](f_1)(x)\big|>\sigma/2\big\}\big)\Big]^{1/p}\\
&+\frac{1}{\theta(u(Q))^{1/p}}\sigma\cdot
\Big[u\big(\big\{x\in Q:\big|[b,I_{\alpha}](f_2)(x)\big|>\sigma/2\big\}\big)\Big]^{1/p}\\
:=&K'_1+K'_2.
\end{split}
\end{equation*}
Applying Theorem \ref{Two2}, the $\mathcal D_\kappa$ condition (\ref{D condition}) of $\theta$ and inequality (\ref{weights})(consider cube $Q$ instead of ball $B$), we get
\begin{equation*}
\begin{split}
K'_1&\leq C\cdot\frac{1}{\theta(u(Q))^{1/p}}\left(\int_{\mathbb R^n}|f_1(x)|^p v(x)\,dx\right)^{1/p}\\
&=C\cdot\frac{1}{\theta(u(Q))^{1/p}}\left(\int_{2Q}|f(x)|^p v(x)\,dx\right)^{1/p}\\
&\leq C\big\|f\big\|_{\mathcal M^{p,\theta}(v,u)}\cdot\frac{\theta(u(2Q))^{1/p}}{\theta(u(Q))^{1/p}}\\
&\leq C\big\|f\big\|_{\mathcal M^{p,\theta}(v,u)}\cdot\frac{u(2Q)^{\kappa/p}}{u(Q)^{\kappa/p}}\\
&\leq C\big\|f\big\|_{\mathcal M^{p,\theta}(v,u)}.
\end{split}
\end{equation*}
Next we estimate $K'_2$. For any $x\in Q$, from the definition of $[b,I_{\alpha}]$, we can see that
\begin{equation*}
\begin{split}
\big|[b,I_{\alpha}](f_2)(x)\big|
&\leq\big|b(x)-b_{Q}\big|\cdot\big|I_\alpha(f_2)(x)\big|
+\Big|I_\alpha\big([b_{Q}-b]f_2\big)(x)\Big|\\
&:=\xi(x)+\eta(x).
\end{split}
\end{equation*}
Consequently, we can further divide $K'_2$ into two parts.
\begin{equation*}
\begin{split}
K'_2\leq&\frac{1}{\theta(u(Q))^{1/p}}\sigma\cdot\Big[u\big(\big\{x\in Q:\xi(x)>\sigma/4\big\}\big)\Big]^{1/p}\\
&+\frac{1}{\theta(u(Q))^{1/p}}\sigma\cdot\Big[u\big(\big\{x\in Q:\eta(x)>\sigma/4\big\}\big)\Big]^{1/p}\\
:=&K'_3+K'_4.
\end{split}
\end{equation*}
For the term $K'_3$, it follows from the pointwise estimate \eqref{alpha1} mentioned above and Chebyshev's inequality that
\begin{equation*}
\begin{split}
K'_3&\leq\frac{4}{\theta(u(Q))^{1/p}}\cdot\left(\int_Q\big|\xi(x)\big|^pu(x)\,dx\right)^{1/p}\\
&\leq\frac{C}{\theta(u(Q))^{1/p}}\cdot\left(\int_Q\big|b(x)-b_{Q}\big|^pu(x)\,dx\right)^{1/p}
\times\bigg(\sum_{j=1}^\infty\frac{1}{|2^{j+1}Q|^{1-{\alpha}/n}}\int_{2^{j+1}Q}|f(y)|\,dy\bigg)\\
&\leq C\cdot\frac{u(Q)^{1/p}}{\theta(u(Q))^{1/p}}
\sum_{j=1}^\infty\frac{1}{|2^{j+1}Q|^{1-{\alpha}/n}}\int_{2^{j+1}Q}|f(y)|\,dy,
\end{split}
\end{equation*}
where in the last inequality we have used the fact that Lemma \ref{BMO}$(ii)$ still holds when $B$ replaced by $Q$ and $u$ is an $A_{\infty}$ weight. Repeating the arguments in the proof of Theorem \ref{mainthm:5}, we can show that $K'_3\leq C\big\|f\big\|_{\mathcal M^{p,\theta}(v,u)}$. As for the term $K'_4$, we can show the following pointwise estimate in the same manner as in the proof of Theorem \ref{mainthm:3}.
\begin{equation*}
\eta(x)=\Big|I_{\alpha}\big([b_{Q}-b]f_2\big)(x)\Big|\leq
C\sum_{j=1}^\infty\frac{1}{|2^{j+1}Q|^{1-{\alpha}/n}}\int_{2^{j+1}Q}\big|b(y)-b_{Q}\big|\cdot\big|f(y)\big|\,dy.
\end{equation*}
This, together with Chebyshev's inequality yields
\begin{equation*}
\begin{split}
K'_4&\leq\frac{4}{\theta(u(Q))^{1/p}}\cdot\left(\int_Q\big|\eta(x)\big|^pu(x)\,dx\right)^{1/p}\\
&\leq C\cdot\frac{u(Q)^{1/p}}{\theta(u(Q))^{1/p}}\cdot
\sum_{j=1}^\infty\frac{1}{|2^{j+1}Q|^{1-{\alpha}/n}}\int_{2^{j+1}Q}\big|b(y)-b_{Q}\big|\cdot\big|f(y)\big|\,dy\\
&\leq C\cdot\frac{u(Q)^{1/p}}{\theta(u(Q))^{1/p}}\cdot
\sum_{j=1}^\infty\frac{1}{|2^{j+1}Q|^{1-{\alpha}/n}}\int_{2^{j+1}Q}\big|b(y)-b_{{2^{j+1}Q}}\big|\cdot\big|f(y)\big|\,dy\\
&+C\cdot\frac{u(Q)^{1/p}}{\theta(u(Q))^{1/p}}\cdot
\sum_{j=1}^\infty\frac{1}{|2^{j+1}Q|^{1-{\alpha}/n}}\int_{2^{j+1}Q}\big|b_{{2^{j+1}Q}}-b_{Q}\big|\cdot\big|f(y)\big|\,dy\\
&:=K'_5+K'_6.
\end{split}
\end{equation*}
An application of H\"older's inequality leads to that
\begin{equation*}
\begin{split}
K'_5&\leq C\cdot\frac{u(Q)^{1/p}}{\theta(u(Q))^{1/p}}\cdot
\sum_{j=1}^\infty\frac{1}{|2^{j+1}Q|^{1-{\alpha}/n}}\left(\int_{2^{j+1}Q}\big|f(y)\big|^pv(y)\,dy\right)^{1/p}\\
&\times\left(\int_{2^{j+1}Q}\big|b(y)-b_{{2^{j+1}Q}}\big|^{p'}v(y)^{-p'/p}\,dy\right)^{1/{p'}}\\
&\leq C\big\|f\big\|_{\mathcal M^{p,\theta}(v,u)}\cdot\frac{u(Q)^{1/p}}{\theta(u(Q))^{1/p}}\cdot
\sum_{j=1}^\infty\frac{\theta(u(2^{j+1}Q))^{1/p}}{|2^{j+1}Q|^{1-{\alpha}/n}}\\
&\times\big|2^{j+1}Q\big|^{1/{p'}}\Big\|(b-b_{{2^{j+1}Q}})\cdot v^{-1/p}\Big\|_{\mathcal C,2^{j+1}Q},
\end{split}
\end{equation*}
where $\mathcal C(t)=t^{p'}$ is a Young function. For $1<p<\infty$, we know the inverse function of $\mathcal C(t)$ is $\mathcal C^{-1}(t)=t^{1/{p'}}$. Observe that
\begin{equation*}
\begin{split}
\mathcal C^{-1}(t)&=t^{1/{p'}}\\
&=\frac{t^{1/{p'}}}{1+\log^+ t}\times\big(1+\log^+t\big)\\
&=\mathcal A^{-1}(t)\cdot\mathcal B^{-1}(t),
\end{split}
\end{equation*}
where
\begin{equation*}
\mathcal A(t)\approx t^{p'}(1+\log^+t)^{p'},\qquad \mbox{and}\qquad \mathcal B(t)\approx e^t-1.
\end{equation*}
Thus, by inequality \eqref{three} and the unweighted version of inequality \eqref{Jensen}(when $w\equiv1$), we have
\begin{equation*}
\begin{split}
\Big\|(b-b_{{2^{j+1}Q}})\cdot v^{-1/p}\Big\|_{\mathcal C,2^{j+1}Q}
&\leq C\Big\|b-b_{{2^{j+1}Q}}\Big\|_{\mathcal B,2^{j+1}Q}\cdot\big\|v^{-1/p}\big\|_{\mathcal A,2^{j+1}Q}\\
&\leq C\|b\|_*\cdot\big\|v^{-1/p}\big\|_{\mathcal A,2^{j+1}Q}.
\end{split}
\end{equation*}
Since $u$ is an $A_{\infty}$ weight, one has $u\in\Delta_2$. Moreover, in view of \eqref{comparethe} and \eqref{U}, we can deduce that
\begin{equation*}
\begin{split}
K'_5&\leq C\|b\|_*\big\|f\big\|_{\mathcal M^{p,\theta}(v,u)}
\sum_{j=1}^\infty\frac{u(2^{j+1}Q)^{\kappa/p}}{u(Q)^{\kappa/p}}\cdot\frac{u(Q)^{1/p}}{|2^{j+1}Q|^{1/p-{\alpha}/n}}
\cdot\big\|v^{-1/p}\big\|_{\mathcal A,2^{j+1}Q}\\
&\leq C\|b\|_*\big\|f\big\|_{\mathcal M^{p,\theta}(v,u)}
\sum_{j=1}^\infty\frac{u(Q)^{{(1-\kappa)}/p}}{u(2^{j+1}Q)^{{(1-\kappa)}/p}}\\
&\times\big|2^{j+1}Q\big|^{{\alpha}/n}\left(\frac{1}{|2^{j+1}Q|}\int_{2^{j+1}Q}u(x)^r\,dx\right)^{1/{(rp)}}
\cdot\big\|v^{-1/p}\big\|_{\mathcal A,2^{j+1}Q}\\
&\leq C\big\|f\big\|_{\mathcal M^{p,\theta}(v,u)}
\sum_{j=1}^\infty\frac{u(Q)^{{(1-\kappa)}/p}}{u(2^{j+1}Q)^{{(1-\kappa)}/p}}\\
&\leq C\big\|f\big\|_{\mathcal M^{p,\theta}(v,u)}.
\end{split}
\end{equation*}
The last inequality is obtained by the $A_p$-type condition \eqref{assump2.1} on $(u,v)$ and the estimate \eqref{C1}. It remains to estimate the last term $K'_6$. Applying Lemma \ref{BMO}$(i)$(use $Q$ instead of $B$) and H\"older's inequality, we get
\begin{equation*}
\begin{split}
K'_6&\leq C\cdot\frac{u(Q)^{1/p}}{\theta(u(Q))^{1/p}}
\sum_{j=1}^\infty\frac{(j+1)\|b\|_*}{|2^{j+1}Q|^{1-{\alpha}/n}}\int_{2^{j+1}Q}|f(y)|\,dy\\
&\leq C\cdot\frac{u(Q)^{1/p}}{\theta(u(Q))^{1/p}}
\sum_{j=1}^\infty\frac{(j+1)\|b\|_*}{|2^{j+1}Q|^{1-{\alpha}/n}}\left(\int_{2^{j+1}Q}\big|f(y)\big|^pv(y)\,dy\right)^{1/p}\\
&\times\left(\int_{2^{j+1}Q}v(y)^{-p'/p}dy\right)^{1/{p'}}\\
&\leq C\big\|f\big\|_{\mathcal M^{p,\theta}(v,u)}\cdot\frac{u(Q)^{1/p}}{\theta(u(Q))^{1/p}} \sum_{j=1}^\infty(j+1)\cdot\frac{\theta(u(2^{j+1}Q))^{1/p}}{|2^{j+1}Q|^{1-{\alpha}/n}}\\
&\times\left(\int_{2^{j+1}Q}v(y)^{-p'/p}dy\right)^{1/{p'}}.
\end{split}
\end{equation*}
Let $\mathcal C(t)$, $\mathcal A(t)$ be the same as before. Obviously, $\mathcal C(t)\leq\mathcal A(t)$ for all $t>0$, then it is not difficult to see that for any given cube $Q\subset\mathbb R^n$, we have $\big\|f\big\|_{\mathcal C,Q}\leq\big\|f\big\|_{\mathcal A,Q}$ by definition, which implies that condition \eqref{assump2.1} is stronger that condition \eqref{assump1.1}. This fact together with \eqref{comparethe} and \eqref{U} yield
\begin{equation*}
\begin{split}
K'_6&\leq C\big\|f\big\|_{\mathcal M^{p,\theta}(v,u)}\sum_{j=1}^\infty(j+1)\cdot
\frac{u(Q)^{{(1-\kappa)}/p}}{u(2^{j+1}Q)^{{(1-\kappa)}/p}}\cdot\frac{u(2^{j+1}Q)^{1/p}}{|2^{j+1}Q|^{1-{\alpha}/n}}\\
&\times\left(\int_{2^{j+1}Q}v(y)^{-p'/p}\,dy\right)^{1/{p'}}\\
&\leq C\big\|f\big\|_{\mathcal M^{p,\theta}(v,u)}\sum_{j=1}^\infty(j+1)\cdot
\frac{u(Q)^{{(1-\kappa)}/p}}{u(2^{j+1}Q)^{{(1-\kappa)}/p}}\cdot\frac{|2^{j+1}Q|^{1/{(r'p)}}}{|2^{j+1}Q|^{1-{\alpha}/n}}\\
&\times\left(\int_{2^{j+1}Q}u(y)^r\,dy\right)^{1/{(rp)}}\left(\int_{2^{j+1}Q}v(y)^{-p'/p}\,dy\right)^{1/{p'}}\\
&\leq C\big\|f\big\|_{\mathcal M^{p,\theta}(v,u)}
\sum_{j=1}^\infty(j+1)\cdot\frac{u(Q)^{{(1-\kappa)}/p}}{u(2^{j+1}Q)^{{(1-\kappa)}/p}}.
\end{split}
\end{equation*}
Moreover, by our additional hypothesis on $u:u\in A_\infty$ and inequality (\ref{compare}) with exponent $\delta>0$(use $Q$ instead of $B$), we finally obtain
\begin{equation*}
\begin{split}
\sum_{j=1}^\infty(j+1)\cdot\frac{u(Q)^{{(1-\kappa)}/p}}{u(2^{j+1}Q)^{{(1-\kappa)}/p}}
&\leq C\sum_{j=1}^\infty(j+1)\cdot\left(\frac{|Q|}{|2^{j+1}Q|}\right)^{{\delta(1-\kappa)}/p}\\
&\leq C\sum_{j=1}^\infty(j+1)\cdot\left(\frac{1}{2^{(j+1)n}}\right)^{{\delta(1-\kappa)}/p}\\
&\leq C,
\end{split}
\end{equation*}
which in turn gives that $K'_6\leq C\big\|f\big\|_{\mathcal M^{p,\theta}(v,u)}$. Summing up all the above estimates, and then taking the supremum over all cubes $Q\subset\mathbb R^n$ and all $\sigma>0$, we therefore conclude the proof of Theorem \ref{mainthm:6}.
\end{proof}
In particular, if we take $\theta(x)=x^\kappa$ with $0<\kappa<1$, then we immediately
get the following two-weight, weak type $(p,p)$ inequalities for $I_{\alpha}$ and $[b,I_{\alpha}]$ in the weighted Morrey spaces.

\begin{corollary}
Let $1<p<\infty$, $0<\kappa<1$ and $0<\alpha<n$. Given a pair of weights $(u,v)$, suppose that for some $r>1$ and for all cubes $Q$, \eqref{assump1.1} holds. If $u\in\Delta_2$, then the fractional integral operator $I_{\alpha}$ is bounded from $\mathcal L^{p,\kappa}(v,u)$ into $W\mathcal L^{p,\kappa}(u)$.
\end{corollary}

\begin{corollary}
Let $1<p<\infty$, $0<\kappa<1$, $b\in BMO(\mathbb R^n)$ and $0<\alpha<n$. Given a pair of weights $(u,v)$, suppose that for some $r>1$ and for all cubes $Q$, \eqref{assump2.1} holds. If $u\in A_\infty$, then the linear commutator $[b,I_{\alpha}]$ is bounded from $\mathcal L^{p,\kappa}(v,u)$ into $W\mathcal L^{p,\kappa}(u)$.
\end{corollary}

\end{document}